\documentclass[11pt]{article}

\usepackage{geometry}
\usepackage[margin=1.5cm,font=footnotesize]{caption}

\usepackage[colorlinks,citecolor=blue]{hyperref}
\usepackage{amsthm} 
\usepackage{amsmath} 
\usepackage{amssymb} 

\usepackage{cleveref}

\usepackage{pinlabel}
\usepackage{calc}
\usepackage{tikz}
\usepackage{enumerate}

\theoremstyle{plain}
\numberwithin{equation}{section}
\numberwithin{figure}{section}
\numberwithin{table}{section}
\newtheorem{theorem}{Theorem}[section]
\newtheorem{corollary}[theorem]{Corollary}
\newtheorem{lemma}[theorem]{Lemma}
\newtheorem{conjecture}[theorem]{Conjecture}
\newtheorem{proposition}[theorem]{Proposition}
\newtheorem*{pennerconst*}{Penner's Construction}

\newtheorem{step}{Step}
\newtheorem*{claim*}{Claim}

\newtheorem{case}{Case}
\newtheorem{proofpart}{Part}

\crefformat{equation}{(#2#1#3)}
\crefmultiformat{equation}{(#2#1#3)}{ and~(#2#1#3)}{, (#2#1#3)}{ and~(#2#1#3)}

\Crefname{step}{Step}{Step}

\newcommand{\Nul}{\operatorname{Nul}}
\newcommand{\calF}{\mathcal{F}}
\newcommand{\calB}{\mathcal{B}}
\newcommand{\calN}{\mathcal{N}}
\newcommand{\RR}{\mathbb{R}}
\newcommand{\CC}{\mathbb{C}}
\newcommand{\ZZ}{\mathbb{Z}}
\newcommand{\QQ}{\mathbb{Q}}
\newcommand{\vv}{{\bf v}}
\newcommand{\uu}{{\bf u}}
\newcommand{\ww}{{\bf w}}
\newcommand{\Mod}{\operatorname{Mod}} 
\newcommand{\PSL}{\operatorname{PSL}} 
\newcommand{\rank}{\operatorname{rank}}
\newcommand{\homeo}{\mathrel{\cong}} 
\newcommand{\hv}{\widehat{\vv}}
\newcommand{\hV}{\widehat{V}}
\newcommand{\Teich}{\mathrm{Teich}}
\newcommand{\e}{\mathbf{e}}
\newcommand\bG{\mathbf{G}}
\newcommand\odd{\mathrm{odd}}
\newcommand\even{\mathrm{even}}

\author{Bal\'azs Strenner}
\date{\today}
\title{Algebraic degrees of pseudo-Anosov stretch factors}

\begin{document}

\maketitle
\abstract{The motivation for this paper is to justify a remark of Thurston that the algebraic degree of stretch factors of pseudo-Anosov maps on a surface $S$ can be as high as the dimension of the Teichm\"uller space of $S$. In addition to proving this, we completely determine the set of possible algebraic degrees of pseudo-Anosov stretch factors on almost all finite type surfaces. As a corollary, we find the possible degrees of the number fields that arise as trace fields of Veech groups of flat surfaces homeomorphic to closed orientable surfaces. Our construction also gives an algorithm for finding a pseudo-Anosov map on a given surface whose stretch factor has a prescribed degree. One ingredient of the proofs is a novel asymptotic irreducibility criterion for polynomials.}

\section{Introduction}
\label{sec:introduction}

Let $S$ be a finite type surface. An element $f$ of the mapping class
group $\Mod(S)$ is \emph{pseudo-Anosov} if there is a representative
homeomorphism $\psi$, a number $\lambda>1$ called the stretch factor
(or dilatation), and a pair of transverse invariant singular measured
foliations $\calF^u$ and $\calF^s$ such that
$\psi(\calF^u) = \lambda\calF^u$ and
$\psi(\calF^s) =\lambda^{-1}\calF^s$. The stretch factor $\lambda$ is
an algebraic integer. The goal of this paper is to determine
the possible algebraic degrees of $\lambda$.

Studying the number $\lambda$ is motivated by its connections to algebraic
geometry, geometric topology and dynamics. For example, $\log(\lambda)$ is the
translation length of $f$ on Teichm\"uller space with the Teichm\"uller metric
and hence the length of a closed geodesic in moduli space \cite[Section
14.2.2]{FarbMargalit12}. The volume of the hyperbolic 3-manifold obtained as
the mapping torus of $f$ is also related to $\log(\lambda)$
\cite{KojimaMcShane14}. Finally, the extension field
$\QQ(\lambda+\lambda^{-1})$ plays a role in Teichm\"uller dynamics
\cite{McMullen03}.

\paragraph{Thurston's remark.}
Thurston announced the classification of elements of $\Mod(S)$ to
finite order, reducible and pseudo-Anosov elements in his seminal
bulletin paper \cite{Thurston88} (which had been circulating as a preprint much earlier). On pages 427--428, he provides a
bound for the algebraic degree of pseudo-Anosov stretch factors
$\lambda$. He denotes the dimension of the Teichm\"uller space of $S$
by $d$ and writes:

\begin{quote}\it
  Therefore $\lambda$ is an algebraic integer of degree $\le d$. The
  examples of Theorem 7 show that this bound is sharp.
\end{quote}

Theorem 7, which describes a construction of pseudo-Anosov mapping
classes using Dehn twists, definitely should produce examples
realizing the degree $d$, but Thurston did not explain this in the
paper, nor did anyone else since then. The intuition that supports
Thurston's claim is that a random degree $d$ polynomial is likely to
be irreducible. However, it is not clear if the polynomials arising
from Thurston's constuction are random in any sense. Another
difficulty is the lack of irreducibility criteria that apply for
defining polynomials of stretch factors. For example, the well-known
Eisenstein's criterion does not apply since it requires the constant term
of the polynomial to be divisible by a prime, whereas the polynomials
in question always have constant term $\pm 1$.

Even the fact that the degree can grow linearly with the genus is nontrivial.
This is due to Arnoux and Yoccoz \cite{ArnouxYoccoz81} who constructed a degree
$g$ stretch factor on each closed orientable surface $S_g$ of genus $g\ge 3$.
More recently Shin \cite{Shin16} also realized the degree $2g$ on $S_g$. (For
$S_g$, the dimension of Teichm\"uller space and hence the maximal degree
predicted by Thurston is $6g-6$.)

\paragraph{The main result.}
In this paper, not only do we realize the theoretical maximum $6g-6$,
but we completely answer the question of which degrees appear on
$S_g$. Moreover, we also answer this question for most finite type
surfaces, including all nonorientable surfaces, for which Thurston's
construction does not apply.

Let $D(S)$ be the set of possible algebraic degrees of stretch factors
of pseudo-Anosov elements of $\Mod(S)$. Let $D^+(S) \subset D(S)$ be
the set of degrees arising from pseudo-Anosov maps with a transversely
orientable invariant foliation.  Finally, denote by $[a,b]_\even$ and
$[a,b]_\odd$ the set of even and odd integers, respectively, in the
interval $[a,b]$.
\begin{theorem}\label{theorem:degrees_simple}
  Let $g\ge 2$. We have
  \begin{displaymath}
    D(S_g) = \big[2,6g-6\big]_\even \cup \big[3,3g-3\big]_\odd
  \end{displaymath}
  and
  \begin{displaymath}
    D^+(S_g) = \big[2,2g\big]_\even \cup \big[3,g\big]_\even.
  \end{displaymath}
\end{theorem}

In fact, we prove a more general result in
\Cref{theorem:degrees-general}, where we also determine $D^+(S)$ for
any finite type surface, and $D(S)$ for
\begin{itemize}
\item nonorientable surfaces with any number of punctures and
\item orientable surfaces with an even number of punctures.
\end{itemize}
We also almost completely determine $D(S)$ for orientable surfaces with an odd
number of punctures. The reason we miss some cases is that these surfaces
are not double covers of nonorientable surfaces. However, our construction
relies on such a covering to realize the highest possible odd degree.

The fact that $D(S_g)$ and $D^+(S_g)$ cannot be larger than stated in
\Cref{theorem:degrees_simple} is well-known. We have
$\min D(S_g) \ge 2$, because only $1$ and $-1$ are algebraic units of
degree 1. Thurston \cite{Thurston88} showed that
$\max D(S_g) \le 6g-6$ and $\max D^+(S_g) \le 2g$, and Long
\cite{Long85} showed that if $d \in D(S)$ is odd, then $d \le 3g-3$. A
similar argument shows that if $d \in D^+(S)$ is odd, then $d \le g$.
(Here and in what follows, we use $d$ to denote any degree, not
necessarily the maximal one.)

\paragraph{Irreducibility of polynomials via converging roots.}

In the work of Arnoux--Yoccoz \cite{ArnouxYoccoz81} and Shin
\cite{Shin16}, a construction of pseudo-Anosov maps is given such that
$\lambda$ is a root of a polynomial which can be shown to be
irreducible\footnote{Irreducibility in this paper is
always meant over $\ZZ$.}. The irreducibility criteria in both cases are specific to
one particular sequence of polynomials and cannot easily be
generalized to realize other degrees.

In this paper we use a novel irreducibility criterion. Since the statement is
elementary and the proof is short, we include the criterion here together with
the proof.

\begin{lemma}\label{lemma:galois-conjugates}
  Let $u_k(x) \in \ZZ[x]$ be a sequence of integral polynomials and let
  $\lambda_k \in \CC$ be a sequence of roots, that is, we assume
  $u_k(\lambda_k)=0$ for all $k$. Suppose that there exists $v(x) \in \CC[x]$
  such that
  \begin{equation}\label{eq:poly-convergence}
    \lim_{k\to \infty}\frac{u_k(x)}{x-\lambda_k} = v(x).
  \end{equation}
  Suppose $v(\theta) = 0$ for some $\theta \in \CC$, and let
  $\theta_k \to \theta$ be a sequence with $u_k(\theta_k) = 0$ for all $k$.
  If $\theta_k \ne \theta$ for all but finitely many $k$, then $\theta_k$ and
  $\lambda_k$ are roots of the same irreducible factor of $u_k(x)$ for all but
  finitely many $k$.
\end{lemma}
\begin{proof}
  Factor the polynomials $u_k(x)$ to irreducible factors over $\ZZ$. Assume for
  a contradiction that $\theta_k$ and $\lambda_k$ are in different irreducible
  factors for infinitely many $k$. By restricting to a subsequence, we may
  assume that $\theta_k$ and $\lambda_k$ are in different irreducible factors
  for all $k$. For all $k$, let $w_k(x) \in \ZZ[x]$ be an irreducible factor
  with root $\theta_k$. Since there is a uniform bound on the absolute values
  of the roots of $w_k(x)$, we can restrict further to a subsequence such that
  all $w_k(x)$ have the same degree and $w_k(x) \to w(x)$ for some $w(x)$.
  Since $w_k(x) \in \ZZ[x]$, we have $w(x) \in \ZZ[x]$ and $w_k(x) = w(x)$ if
  $k$ is large enough. In particular, $\theta_k = \theta$ if $k$ is large
  enough, which is a contradiction.
\end{proof}

The criteria are cleanest to state and prove for sequences of
polynomials, but quantitative versions would also be possible to
obtain using Newton's formulas for coefficients of polynomials in
terms of the roots.

\paragraph{Finding stretch factors with prescribed degrees.}
To construct stretch factors of prescribed algebraic degrees, we use
Penner's construction of pseudo-Anosov maps. In this construction one
has two choices: the choice of a collection of curves and the choice of
a product of Dehn twists about these curves. Interestingly, the
algebraic degree of a stretch factor arising from this construction
seems to depend primarily only on the choice of the collection of curves,
and not the Dehn twist product. More specifically, computer
experiments have led to the following observation.

\begin{quote}\it
  The algebraic degree of a stretch factor arising from Penner's
  construction typically equals the rank of the intersection matrix of
  the collection of curves.
\end{quote}

However, the choice of the Dehn twist product also has some effect on
the algebraic degree of the stretch factor. Unfortunate choices of
Dehn twist products may result in a lower algebraic degree than the
typical one.

On the other hand, we will show that for certain infinite sequences of
Dehn twist products the above observation is guaranteed to hold
asymptotically. This way we obtain a simple criterion
(\Cref{theorem:curves-and-degree}) stating that if a collection of curves with
rank $d$ intersection matrix exists on a surface $S$, then
$d \in D(S)$. With this criterion in hand, the problem of realizing
algebraic degrees reduces to the problem of constructing collections
of curves on surfaces with certain properties. We construct
collections of curves in \Cref{sec:collections_of_curves}.

\paragraph{Constructing sequences of polynomials with converging roots.}

In order to prove \Cref{theorem:curves-and-degree}, we construct sequences
$f_1, f_2, \ldots$ of pseudo-Anosov mapping classes such that the defining
polynomials $u_k(x)$ of the stretch factors satisfy the hypotheses of
\Cref{lemma:galois-conjugates} and relate the rank of the intersection matrix
to the number of disjoint sequences $\theta_k \to \theta$ where
$\theta_k \ne \theta$ for all but finitely many $k$.

The sequences $f_k$ are constructed as follows. Fixing a collection of curves
in Penner's construction, we start with some product of Dehn twists and we
modify this product for each $k$ by replacing the Dehn twists with their $k$th
power. It turns out that the defining polynomials of such sequences have the
asymptotic behavior as in \Cref{lemma:galois-conjugates}. We prove this in
three parts, presented in
\Cref{sec:convergence,sec:homotopy-invariance,sec:rank-criterion}.

The fact that the defining polynomials converge in the sense of
\Cref{eq:poly-convergence} is shown in \Cref{theorem:convergence}.
This has two main parts: showing that the left Perron--Frobenius
eigenvectors of certain matrices associated to the Dehn twist products
converge (\Cref{sec:left_PF_eigenvector_estimate}) and showing that
the left actions of the matrices on the orthogonal complements of these
eigenvectors also converge (\Cref{sec:convergence_of_linear_maps}).

In \Cref{sec:homotopy-invariance}, we show that the limit of the left actions
in the previous paragraph turns out to be a composition of projections from
hyperplanes in $\RR^n$ to other hyperplanes in $\RR^n$. Moreover, for
appropriate choices of Dehn twist products, cancellations occur in this
composition of projections, so the limit of the left actions \emph{is} a
projection and therefore we have $v(x) = x(x-1)^s$ for the limit polynomial in
\Cref{lemma:galois-conjugates}. Therefore we will use
\Cref{lemma:galois-conjugates} with $\theta = 0$ or $1$.

Finally, \Cref{prop:multiplicity-of-one} relates the rank of the
intersection matrix to the number of roots of $u_k(x)$ that are
different from 1. The roots are also always different from 0. This
gives a count for sequences $\theta_k \to \theta$ where
$\theta_k \ne \theta$ for all $k$.

\paragraph{Realizing degrees algorithmically.}
Our approach also provides an algorithm for finding a pseudo-Anosov mapping
class on a given surface whose stretch factor has a prescribed algebraic
degree. Indeed, the pseudo-Anosov mapping classes $f_k$ described above have
stretch factors that eventually have the prescribed degree. Therefore one can
iterate over $f_1,f_2,\ldots$ to find a desired example in finite time.

The Dehn twist products for which the degree of the stretch factor is not the
rank of the intersection matrix seem to be rare, so in practice $f_1$ is very
often already a good example. However, it would be interesting to prove a bound
on the smallest $k$ such that $f_k$ is guaranteed to have a stretch factor with
the prescribed degree. Such a bound seems attainable by effectivizing
\Cref{lemma:galois-conjugates} and estimating the rate of convergence in
\Cref{theorem:convergence}. Not only would such a bound give an estimate on the
running time of the algorithm, but it would also allow one to give formulas for
mapping classes whose stretch factors have prescribed algebraic degrees.

\paragraph{Degrees of trace fields of Veech groups.}

Our second result concerns trace fields of Veech groups. A
half-translation surface is a surface with a singular Euclidean
structure with trivial or $\ZZ_2$-holonomy. Its Veech group is the
group of its $\PSL(2,\RR)$-symmetries. Every Veech group is a Fuchsian
group, and its trace field is a natural invariant of the
half-translation surface. There is a half-translation surface
associated with every pseudo-Anosov map $f$ which is defined by the
stable and unstable foliations of $f$. The trace field of the Veech
group of this surface is $\QQ(\lambda + \lambda^{-1})$, where
$\lambda$ is the stretch factor of $f$. The degree of the field
extension $\QQ(\lambda + \lambda^{-1}) : \QQ$ is either the algebraic
degree of $\lambda$ or half of it. For more details, see \cite[\S
11.3]{FarbMargalit12}, \cite{Zorich06}, \cite[\S 7]{KenyonSmillie00},
\cite[\S 7]{GutkinJudge00} or \cite[\S 9]{McMullen03}.

Which Fuchsian groups arise as Veech groups is an open question
\cite[Problem 5]{HubertMasurSchmidtZorich06}. Whether there is a
cyclic Veech group generated by a hyperbolic element is also unknown
\cite[Problem 6]{HubertMasurSchmidtZorich06}. Also little is known
about the number fields that arise as trace fields of Veech groups. As
a corollary of our results on the algebraic degrees of stretch
factors, we obtain the following.

\begin{theorem}\label{theorem:trace_fields}
  The set of degrees of number fields that arise as trace fields of
  Veech groups of half-translation surfaces homeomorphic to $S_g$ is
  $\{1,\ldots,3g-3\}$.
\end{theorem}

\paragraph{Pseudo-Anosov mapping classes that are not lifts.}
David Futer and Samuel Taylor has pointed out to us the following corollary of
\Cref{theorem:degrees_simple}.

\begin{corollary}\label{cor:not-lift}
  For any $g \ge 2$, there exists a pseudo-Anosov element of $\Mod(S_g)$ that has
  no power that arises by lifting a pseudo-Anosov mapping class on a lower
  genus surface by a branched covering.
\end{corollary}
\begin{proof}
  Any pseudo-Anosov mapping class whose stretch factor has degree $6g-6$ has
  the required property. This is because the degree is preserved under taking
  powers (\Cref{lemma:degree-under-powers}) and it is clearly also preserved
  under lifts. However, $6g-6$ is not a possible degree on lower genus surfaces.
\end{proof}

Bestvina and Fujiwara have also described a property of pseudo-Anosov maps such
that if this property is satisfied, then no power of the map is a lift by a
branched covering \cite[Lemma 6.2]{BestvinaFujiwara17}. In Example 6.4 of their
paper, they build an explicit example in genus 3 satisfying with the above
property. Earlier, Bestvina and Fujiwara also proved a result analogous to
\Cref{cor:not-lift} for unbranched coverings instead of branched coverings
\cite[Proposition 4.2]{BestvinaFujiwara07}.

We remark that \Cref{cor:not-lift} can easily be generalized to punctured and
nonorientable surfaces also using the more general \Cref{theorem:degrees-general}
instead of \Cref{theorem:degrees_simple}.

\paragraph{Open questions.}
Our construction uses Penner's construction, not Thurston's construction,
therefore Thurston's remark that the maximal degree arises from his
construction is yet to be justified. We note that it is possible that his remark
applies only to orientable surfaces, because we are not aware of a natural
adaptation of his construction to nonorientable surfaces.

The algebraic degree of $\lambda$ is an interesting measure of
complexity of $f$. Franks and Rykken \cite{FranksRykken99} (see also
\cite[Theorem 5.5]{GutkinJudge00}) showed that if $S$ is orientable
and $\calF^u$ and $\calF^s$ are transversely orientable, then $f$ is a
lift of an Anosov mapping class of the torus by a branched covering if
and only if the degree of $\lambda$ is 2. Farb conjectured that this
phenomenon generalizes to higher degrees.

\begin{conjecture}[Farb]\label{conj:farb}
  Given any $d$ there exists $h(d)$ so that any pseudo-Anosov map with
  degree $d$ stretch factor on a closed orientable surface arises by
  lifting a pseudo-Anosov map on some surface of genus at most $h(d)$
  by a (branched or unbranched) cover.
\end{conjecture}

However, this generalization turns out to be false. Leininger and Reid
\cite{LeiningerReid17} and independently Yazdi \cite{Yazdi17} have recently
announced that they have counterexamples to \Cref{conj:farb}.

There are many other open questions about the degrees of stretch
factors. Margalit asked what the possible algebraic degrees of stretch
factors in the Torelli group are. Computer experiments suggest that
the same degrees occur in the Torelli group as in the whole mapping
class group. We wonder if the methods of this paper can be used to
prove this. One can ask the same question for any other subgroup of
$\Mod(S)$. For the point-pushing subgroup, it would be interesting to
know if the degree of the stretch factor is related to some property
of the corresponding element of the fundamental group.

It is also not known what degrees are generic in $\Mod(S)$, its subgroups or strata of the holomorphic quadratic differentials over
the moduli space of $S$. We conjecture that the largest possible degree ($6g-6$ in the case of $S_g$) is generic in the whole mapping class group. In the non-maximal strata, the degree cannot reach $6g-6$, and the likely scenario is that the generic degree in each stratum is the maximal degree that can be realized in that stratum.

Finally, many of these questions have versions for outer automorphisms of free groups.

\paragraph{The structure of the paper.}

In \Cref{sec:penners-construction-background} we review some basic facts about
Penner's construction. In
\Cref{sec:convergence,sec:homotopy-invariance,sec:rank-criterion} we prove the
three parts contributing to the proof of \Cref{theorem:curves-and-degree} which
reduces the problem of realizing degrees to realizing ranks of intersection
matrices of collections of curves. We construct collections of curves in
\Cref{sec:collections_of_curves}. In two cases, we fail to construct a
collection of curves whose intersection matrix has the desired rank. In these
two cases, we give explicit examples of pseudo-Anosov mapping classes in
\Cref{sec:two_examples} without using Penner's construction. The proofs of the
main theorems are given in \Cref{sec:proofs}.

\paragraph{Acknowledgements.} This work is an improvement over a part
of the Ph.D.~thesis of the author. He thanks his advisor, Autumn Kent,
for her guidance and Dan Margalit for suggesting the problem and for
many valuable comments. He also appreciates helpful conversations with
Joan Birman, Jordan Ellenberg, Benson Farb, Eriko Hironaka and
Jean-Luc Thiffeault. Special thanks to the referees whose comments
have greatly helped improve the clarity of the exposition.

The author was partially supported by the grant NSF DMS-1128155.

\section{Background on Penner's construction}
\label{sec:penners-construction-background}

In this section we recall Penner's construction and some facts about the
construction that will be used later in the paper.

\subsection{Penner's construction}
\label{sec:penners-construction}

Consider the annulus $A = S^1 \times [0,1]$. We orient $A$ via its embedding to
the $(\theta,r)$-plane (the plane parametrized in polar coordinates) by the map
$(\theta,t) \mapsto (\theta,t+1)$. The orientation of $A$ is defined to be
consistent with the standard orientation of the plane. The standard Dehn twist
$T:A \to A$ is defined by the formula $T(\theta,t) = (\theta+2\pi t,t)$.

Let $c$ be a two-sided simple closed curve on a surface $S$, and let
$\phi: A \to S$ be a homeomorphism between $A$ and a regular
neighborhood of $c$. We refer to the pair $(c,\phi)$ as a \emph{marked
  curve}. The \emph{Dehn twist} about the marked curve $(c,\phi)$ is the
homeomorphism $T_{c,\phi}$ defined by the formula
\begin{displaymath}
  T_{c,\phi}(x) =
  \begin{cases}
    \phi \circ T \circ \phi^{-1}(x) & \mbox{if } x\in \phi(A) \\
    x & \mbox{if } x \in S-\phi(A). \\
  \end{cases}
\end{displaymath}
In the rest of the paper we will not distinguish between $T_{c,\phi}$
and its mapping class. Note that if $S$ is oriented, then $T_{c,\phi}$
is the left Dehn twist $T_c$ if $\phi$ is orientation-preserving and
the right Dehn twist $T_c^{-1}$ otherwise.

Let $(c,\phi_c)$ and $(d,\phi_d)$ be marked curves that intersect at a
point $p$. We say that they are \emph{marked inconsistently} at $p$ if
the pushforward of the orientation of $A$ by $\phi_c$ and $\phi_d$
disagree near $p$.

Two simple closed curves on a surface are in \emph{minimal position}
if they realize the minimal intersection number in their homotopy
classes. A collection of simple closed curves $C$ on a surface is
\emph{filling} if the curves are in pairwise minimal position and the
components of $S-C$ are disks or once-punctured disks.

Penner gave the following construction for pseudo-Anosov mapping
classes \cite{Penner88}. See also \cite{Fathi92} for a different proof.

\begin{pennerconst*}[General case]\label{theorem:penners_construction}
  Let $C = \{(c_1,\phi_1), \ldots, (c_n,\phi_n)\}$ be a filling collection
  of marked curves on $S$. Suppose that they are marked inconsistently at
  every intersection.   Then any product of the $T_{c_i,\phi_i}$ is pseudo-Anosov
  provided each twist is used at least once.
\end{pennerconst*}

When $S$ is orientable, the hypotheses imply that $C$ is a union
of two multicurves $A$ and $B$, and the statement takes the following
more well-known form.

\begin{pennerconst*}[Orientable case]
  Let $A = \{a_1, \ldots, a_n\}$ and $B = \{b_1, \ldots, b_m\}$ be a
  pair of filling multicurves on an orientable surface $S$. Then any
  product of $T_{a_j}$ and $T_{b_k}^{-1}$ is pseudo-Anosov provided
  that each twist is used at least once.
\end{pennerconst*}

We remark that if $C$ is a union of a pair of multicurves and the
curves are marked inconsistently at every intersection, then the
surface filled by $C$ is necessarily orientable. Hence if $S$ is
nonorientable, then $C$ in Penner's construction cannot be a union
of two multicurves.

\subsection{Oriented collections of marked curves}
\label{sec:oriented_collections}

Let $(c,\phi)$ be a marked curve. We define its \emph{left side} as
$\phi(S^1 \times \{0\})$ and its \emph{right side} as
$\phi(S^1 \times \{1\})$. Note that the marking $\phi$ also induces an
orientation of $c$ from the standard (counterclockwise) orientation of
$S^1$.

Suppose $(c,\phi_c)$ and $(d,\phi_d)$ are marked inconsistently at
some $p \in c\cap d$. When we follow $c$ in the direction of its
orientation near $p$, we either cross from the left side of $d$ to the
right side of $d$ or the other way around. In the first case, we call
$p$ a \emph{left-to-right} crossing. In the second case, we call it
\emph{right-to-left} crossing. Note that the definition is symmetric
in $c$ and $d$: if $c$ crosses from the left side of $d$ to the
right side of $d$, then $d$ also must cross from the left side of $c$
to the right side of $c$ (\Cref{fig:crossings}).

\begin{figure}[ht]
  \centering
  \begin{tikzpicture}[thick]
    \draw[->,red] (-2,0) -- (2,0);
    \node[above] at (1.5,0) {$R$};
    \node[below] at (1.5,0) {$L$};
    \draw[->,blue] (0,-2) -- (0,2);
    \node[left] at (0,1.5) {$L$};
    \node[right] at (0,1.5) {$R$};

    \begin{scope}[xshift = 5cm]
      \draw[->,red] (2,0) -- (-2,0);
      \node[above] at (1.5,0) {$L$};
      \node[below] at (1.5,0) {$R$};
      \draw[->,blue] (0,-2) -- (0,2);
      \node[left] at (0,1.5) {$L$};
      \node[right] at (0,1.5) {$R$};
    \end{scope}
  \end{tikzpicture}
  \caption{A left-to-right and a right-to-left crossing.}
  \label{fig:crossings}
\end{figure}
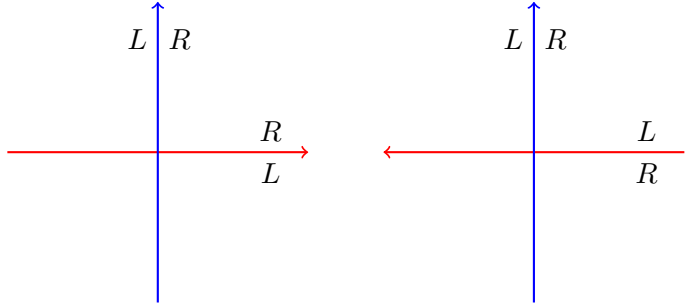

Let $C$ be a collection of marked curves which are inconsistently
marked at each intersection. Then $C$ is called \emph{completely
  left-to-right} if all crossings are left-to-right and
\emph{completely right-to-left} if all crossings are right-to-left.

We also recall that a singular foliation on a surface is \emph{transversely
  orientable} if there is a continuous choice of vectors at the non-singular
points of the foliation so that at each point the chosen vector is not tangent
to the leaf of the foliation going through that point.

\begin{proposition}\label{prop:train-track-orientable-foliation}
  If $C$ is completely left-to-right or completely right-to-left, then
  the pseudo-Anosov maps constructed from it by Penner's construction
  have a transversely orientable invariant foliation.
\end{proposition}
\begin{proof}
  Penner \cite[p.~188]{Penner88} observed that by smoothing out the intersections of $C$, one
  obtains a bigon track $\tau^+$ invariant under the Dehn twists of
  $C$ and hence under any pseudo-Anosov map $\psi$ constructed from
  his construction. By choosing the smoothings differently at every
  intersection, we get a track $\tau^-$ invariant under $\psi^{-1}$.
  The unstable foliation is carried by $\tau^+$, and the stable
  foliation is carried by $\tau^-$.

\begin{figure}[ht]
  \centering
  \begin{tikzpicture}[thick]
    \draw[red] (-2,0) -- (-1,0) .. controls (-0.25,0) and (0,0.5) ..
    (0,0);
    \draw[red,->] (0,0) .. controls (0,-0.5) and (0.25,0) .. (1,0) -- (2,0);
    \node[above] at (1.5,0) {$R$};
    \node[below] at (1.5,0) {$L$};
    \draw[->,blue] (0,-2) -- (0,2);
    \node[left] at (0,1.5) {$L$};
    \node[right] at (0,1.5) {$R$};
    \draw[very thick,->] (1,0) -- +(0,0.5);
    \draw[very thick,->] (-1,0) -- +(0,0.5);
    \draw[very thick,->] (0,1) -- +(0.5,0);

    \begin{scope}[xshift = 5cm]
      \draw[red] (-2,0) -- (-1,0) .. controls (-0.25,0) and (0,-0.5) ..
      (0,0);
      \draw[red] (0,0) .. controls (0,0.5) and (0.25,0) .. (1,0) -- (2,0);
      \draw[->,red] (-1,0) -- (-2,0);
      \node[above] at (1.5,0) {$L$};
      \node[below] at (1.5,0) {$R$};
      \draw[->,blue] (0,-2) -- (0,2);
      \node[left] at (0,1.5) {$L$};
      \node[right] at (0,1.5) {$R$};
    \draw[very thick,->] (1,0) -- +(0,-0.5);
    \draw[very thick,->] (-1,0) -- +(0,-0.5);
    \draw[very thick,->] (0,1) -- +(0.5,0);
    \end{scope}
  \end{tikzpicture}
  \caption{Transverse orientation in a neighborhood of a crossing.
    A left-to-right crossing and $\tau^-$ on the left. A right-to-left
    crossing and $\tau^+$ on the right.}
\label{fig:smoothing_out}
\end{figure}
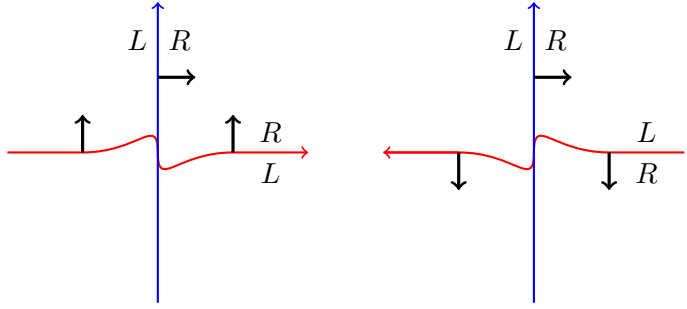

  When $C$ is completely left-to-right, $\tau^-$ is transversely
  orientable: there is a continuously varying set of vectors
  transverse to $\tau^-$ such that the vectors point toward the right
  side of the curves of $C$ (\Cref{fig:smoothing_out}). Hence the
  stable foliation is transversely orientable.

  Similarly, when $C$ is completely right-to-left, $\tau^+$ is
  transversely orientable, and the unstable foliation is transversely
  orientable.
\end{proof}

\subsection{Description Penner's construction by linear algebra}

Denote by $i(a,b)$ the geometric intersection number of the simple
closed curves $a$ and $b$. For two collections of curves $A = \{a_j\}$
and $B = \{b_k\}$, the intersection matrix $i(A,B)$ is a matrix whose
$(j,k)$-entry is $i(a_j,b_k)$.

Suppose we apply Penner's construction with the collection
$C = \{(c_1,\phi_1),\ldots,(c_n,\phi_n)\}$. Let
\begin{displaymath}
\Omega = i(C,C)
\end{displaymath}
and
\begin{equation}\label{eq:Q_i}
  Q_i = I + D_i\Omega \quad (1\le i \le n),
\end{equation}
where $I$ denotes the $n \times n$ identity matrix, and $D_i$ denotes
the $n \times n$ matrix whose $i$th entry on the diagonal is 1 and
whose other entries are zero. 
When a product of the twists $T_{c_i,\phi_i}$ satisfies the hypotheses
of Penner's construction, the corresponding product of the $Q_i$ is
\emph{Perron--Frobenius}, that is, it has nonnegative entries and some
power of it has strictly positive entries \cite[Proposition
5.3]{Fathi92}. Moreover, the stretch factor of the resulting
pseudo-Anosov mapping class equals the Perron--Frobenius eigenvalue
(the unique largest real eigenvalue) of the corresponding
Perron--Frobenius matrix \cite[Th\'eor\`eme 5.4]{Fathi92}. Our
matrices $Q_i$ are the transposes of $M(T_j)$ and $M(S_j^{-1})$ in
Section 5 of \cite{Fathi92}.

\section{Sequences of pseudo-Anosov mapping classes}
\label{sec:convergence}

In this section we study the asymptotic behavior of certain sequences of
pseudo-Anosov mapping classes arising from Penner's construction. We show that
the defining polynomials of the stretch factors converge in the sense of
\Cref{eq:poly-convergence}. Before we state the theorem more precisely, we need
to introduce some notation.

Let $\Omega$ be the intersection matrix of a filling collection of curves $C$. Let
$Z_i$ be the orthogonal complement of the $i$th row of $\Omega$. Since each
curve in $C$ intersects some other curve in $C$, all rows of $\Omega$ are
nonzero, and the $Z_i$ are hyperplanes. Let
\begin{equation}\label{eq:p-i-gets-j}
  p_{i \gets j} : \RR^n \to Z_i
\end{equation}
be the (not necessarily orthogonal) projection onto the hyperplane
$Z_i$ in the direction of $\e_j$, the $j$th standard basis vector in
$\RR^n$. This projection is defined if and only if $\e_j$ is not
contained in $Z_i$, which is in turn equivalent to the statement that
the $(i,j)$-entry of $\Omega$ is positive.

Let $\bG(\Omega)$ be the graph on the vertex set $\{1, \ldots, n\}$
where $i$ and $j$ are connected if the $(i,j)$-entry of $\Omega$ is
positive. For a closed path
\begin{displaymath}
  \gamma = (i_1\cdots i_Ki_1)
\end{displaymath}
in
$\bG(\Omega)$, define the linear map $f_\gamma: Z_{i_1} \to Z_{i_1}$
by the formula
\begin{equation}\label{eq:f-gamma}
  f_\gamma = (p_{i_1\gets i_{K}}\circ \cdots
  \circ p_{i_2\gets i_1})|_{Z_{i_1}}.
\end{equation}
In words, $f_\gamma$ is a composition of projections: first from
$Z_{i_1}$ to $Z_{i_2}$, then from $Z_{i_2}$ to $Z_{i_3}$, and so on, and finally
from $Z_{i_K}$ back to $Z_{i_1}$.

\begin{theorem}\label{theorem:convergence}
  Let $\Omega$ be the intersection matrix of a collection of curves
  satisfying the hypotheses of Penner's construction. Let
  $\gamma = (i_1\ldots i_Ki_1)$ be a closed path in $\bG(\Omega)$
  visiting each vertex at least once and let
  \begin{displaymath}
    M_{\gamma,k} = Q_{i_K}^{k} \cdots Q_{i_1}^{k}.
  \end{displaymath}
  Let $\lambda_k$ be the Perron--Frobenius eigenvalue of
  $M_{\gamma,k}$ and denote by $u_k(x)$ and $v(x)$ the characteristic
  polynomials $\chi(M_{\gamma,k})$ and $\chi(f_\gamma)$, respectively.
  Then we have
  \begin{displaymath}
    \lim_{k\to \infty}\frac{u_k(x)}{x-\lambda_k} = v(x).
  \end{displaymath}
\end{theorem}

The rest of the section is devoted to the proof of this statement.

\subsection{Estimating the left Perron--Frobenius eigenvectors}
\label{sec:left_PF_eigenvector_estimate}

In this section we study the asymptotic behavior of the left Perron--Frobenius
eigenvectors of the matrices $M_{\gamma,k}$ in \Cref{theorem:convergence}.
First we introduce a few conventions.

We follow the convention that vectors denoted by bold lowercase letters are
column vectors. Row vectors are written as transposes of column vectors. Note
that the $i$th row of a matrix $M$ can be written as $\e_i^T M$. We will write
$\omega_{ij}$ for the $(i,j)$-entry of $\Omega$.

We recall that the Perron--Frobenius eigenvectors of a Perron--Frobenius matrix
are the eigenvectors corresponding to the Perron--Frobenius eigenvalue whose
coordinates are positive. The Perron--Frobenius eigenvectors form a ray
emanating from the origin. When we say that the Perron--Frobenius eigenvector
can be chosen to have some property, we mean that there is an eigenvector on
this ray satisfying that property.

The result of this section is the following.

\begin{proposition}\label{prop:eigenvector_asymptotics}
  Let $\Omega$ be the intersection matrix of a collection of curves
  satisfying the hypotheses of Penner's construction. Let
  $\gamma = (i_1\ldots i_Ki_1)$ be a closed path in $\bG(\Omega)$
  visiting each vertex at least once and let
  \begin{displaymath}
    M_{\gamma,k} = Q_{i_K}^{k} \cdots Q_{i_1}^{k}.
  \end{displaymath}
  Then the left Perron--Frobenius eigenvector $\ww_k^T$ of
  $M_{\gamma,k}$ can be chosen for all $k$ so that
  \begin{equation}\label{eq:limit_of_difference}
    \lim_{k\to \infty} \left[\ww_k^T -
      \left(k\e_{i_1}^T\Omega +
        \omega_{i_2i_1}^{-1}\e_{i_2}^T\Omega\right)\right] = 0.
  \end{equation}
\end{proposition}
\begin{proof}
  Note that we can assume that $\gamma$ visits every vertex
  of $\bG(\Omega)$ at least twice. This is because the square of
  $M_{\gamma,k}$ is associated to a path visiting every vertex at
  least twice, has the same Perron--Frobenius eigenvectors as
  $M_{\gamma,k}$, and the quantities in
  \Cref{eq:limit_of_difference} are the same for $M_{\gamma,k}$
  and its square.

  Let $||\Omega||_\infty$ be the maximum of the entries of $\Omega$
  and let $||\Omega||_{\min,\gamma}$ be the minimum of those entries
  of $\Omega$ that are at a position $(i,j)$ such that $ij$ is an edge
  in $\gamma$. We will show that
  \begin{equation}\label{eq:row-estimate}
    ||\ww_{k}^T - k\e_{i_{1}}^T\Omega -
    \omega_{i_2i_1}^{-1}\e_{i_2}^T\Omega ||_\infty \le
    \frac{2^{K}||\Omega||_{\infty}^{K-1}}{k||\Omega||_{\min,\gamma}^{K}},
  \end{equation}
  which clearly implies the statement of the proposition. We break the
  proof of this inequality to two steps.

  \begin{step}\label{step:eigenvector-in-cone}
    The left Perron--Frobenius eigenvectors of $M_{\gamma,k}$ are contained in
    the cone generated by the rows of $\Omega M_{\gamma,k}$.
  \end{step}
  \begin{proof}
    First we recall a geometric proof of the fact that every Perron--Frobenius
    matrix has a (left) eigenvector in the positive cone $\RR^{n}_{\ge0}$. The
    key idea is that the right action of the matrix fixes the positive cone and
    it maps rays to rays. So we have a continuous action on the space of rays
    in the positive cone. This space is homeomorphic to a disk, so by the
    Brouwer fixed point theorem there is a fixed ray. This fixed ray corresponds
    to an eigenvector \cite{BorobiaTrias92}.

    We are going to run the same argument not for the positive cone but the
    smaller cone generated by the rows of $\Omega M_{\gamma,k}$. For this, we
    need to prove that this cone is invariant under the right action of
    $M_{\gamma,k}$. With that in hand, we obtain that $M_{\gamma,k}$ has an eigenvector
    in this smaller cone. It is well-known that the only eigenvectors in the
    positive cone are the Perron--Frobenius eigenvectors, so this implies the
    statement to be proven.

    Now we turn to showing that the cone $C_k$ generated by the rows of
    $\Omega M_{\gamma,k}$ is invariant under the right action of
    $M_{\gamma,k}$. Let
    \begin{displaymath}
      C = \{\vv^T \Omega: \vv \ge 0\}
    \end{displaymath}
    be the cone generated by the rows of $\Omega$. Note that
    $C_k = CM_{\gamma,k}$. It suffices to show that $C$ is invariant
    under the right action of $M_{\gamma,k}$ for all $k$, that is, that
    $CM_{\gamma,k} \subset C$. This is because multiplying both sides
    by $M_{\gamma,k}$ gives $CM^2_{\gamma,k} \subset CM_{\gamma,k}$,
    and substituting $C_k = CM_{\gamma,k}$ yields
    $C_kM_{\gamma,k} \subset C_k$.

    To show that $C$ is invariant under the right action of
    $M_{\gamma,k}$ for all $k$, it suffices to show that $C$ is
    invariant under the right action of the generators $Q_i$. For
    this, observe that
    \begin{displaymath}
      \Omega Q_i = \Omega (I+D_i\Omega) = (I + \Omega D_i) \Omega =
      Q_i^T \Omega
    \end{displaymath}
    for $1\le i \le n$. Finally note that for any $\vv \ge 0$ and
    $\vv^T\Omega \in C$, we have
    $(\vv^T\Omega)Q_i = (\vv^T Q_i^T)\Omega \in C$, since $\vv \ge 0$ implies
    $\vv^T Q_i^T\ge 0$.
  \end{proof}

  \begin{step}\label{step:row-estimate}
    For all $k$, every row of $\Omega M_{\gamma,k}$ can be normalized
    to a vector $\uu^T$ (that depends on $k$ and the row) that satisfies
    \begin{displaymath}
      ||\uu^T - k\e_{i_{1}}^T\Omega -
      \omega_{i_2i_1}^{-1}\e_{i_2}^T\Omega ||_\infty \le
      \frac{2^{K}||\Omega||_{\infty}^{K-1}}{k||\Omega||_{\min,\gamma}^{K}}.
    \end{displaymath}
  \end{step}

  \begin{proof}
    By expanding $M_{\gamma,k}$ and the $Q_i$ out using their definitions, we obtain
    \begin{align*}
      \Omega M_{\gamma,k} =\Omega Q_{i_{K}}^{k}
      \cdots Q_{i_{1}}^{k} &= \Omega (I +
                             kD_{i_{K}}\Omega) \cdots (I +
      kD_{i_1}\Omega) = \\
      &=\Omega +\sum_{1 \le t_1 < \ldots < t_\ell \le K}k^\ell \Omega D_{i_{t_\ell}}\Omega \cdots \Omega
      D_{i_{t_1}} \Omega  = \\
      &=\Omega + \sum_{i=1}^{n} \sum_{\stackrel{1 \le t_1 < \ldots <
          t_\ell \le K}{i_{t_\ell} = i}} k^\ell\Omega
      D_{i_{t_\ell}}\Omega \cdots \Omega D_{i_{t_1}} \Omega.
    \end{align*}
    So the $i'$th row of $\Omega M_{\gamma,k}$ is
    \begin{equation}\label{eq:i_primeth_row}
      \e_{i'}^T\Omega M_{\gamma,k}  =
      \e_{i'}^T\Omega + \sum_{i=1}^{n} H(i,i'),
    \end{equation}
    where
    \begin{equation}\label{eq:H_i_i}
      H(i,i') = \sum_{\stackrel{1 \le t_1 < \ldots < t_\ell \le K}{i_{t_\ell}
          = i}} k^\ell \e_{i'}^T\Omega
      D_{i_{t_\ell}}\Omega \cdots \Omega D_{i_{t_1}} \Omega.
    \end{equation}
    Right multiplication by $D_j$ zeroes out all columns except the
    $j$th column, therefore the following identity holds:
    \begin{displaymath}
      \e_i^T \Omega D_j = \omega_{ij}\e_j^T.
    \end{displaymath}
    Repetitively applying this identity for the terms in
    \Cref{eq:H_i_i}, we obtain
    \begin{equation}\label{eq:H_i_i_two}
      H(i,i') = \sum_{\stackrel{1 \le t_1 < \ldots < t_\ell \le K}{i_{t_\ell}
          = i}} k^\ell \omega_{i'i_{t_\ell}}\omega_{i_{t_\ell}i_{t_{\ell-1}}}\cdots
      \omega_{i_{t_2}i_{t_1}} \e_{i_{t_1}}^T\Omega.
    \end{equation}
    A summand on the right hand side is nonzero if and only if the
    path $(i_{t_1}\ldots i_{t_\ell} i')$ is contained $\bG(\Omega)$.
    In particular, if $ii' \notin \bG(\Omega)$, then all summands
    vanish and $H(i,i') = 0$.

    Now suppose that $ii' \in \bG(\Omega)$. Let $t$ be the largest element of
    $\{1,\ldots,K\}$ such that $i_t = i$. Since $\gamma$ visits $i$ at least
    twice, such $t$ exists and $t\ge 3$. Since $(i_1\ldots i_ti')$ and
    $(i_2\ldots i_ti')$ are paths in $\bG(\Omega)$, the right hand side of
    \Cref{eq:H_i_i_two} contains the nonzero summands
    \begin{displaymath}
      k^t\omega_{i'i_{t}}
      \cdots \omega_{i_{2}i_1} \e_{i_{1}}^T\Omega,
    \end{displaymath}
    \begin{displaymath}
      k^{t-1}\omega_{i'i_{t}}
      \cdots \omega_{i_{3}i_2} \e_{i_{2}}^T\Omega.
    \end{displaymath}
    Moreover, these two paths are the unique longest and second
    longest paths of the form $(i_{t_1}\ldots i_{t_\ell}i')$ in
    $\bG(\Omega)$ satisfying $i_{t_\ell} = i$. So we have
    \begin{equation}\label{eq:H_i_i_formula}
      H(i,i') = c(i,i')\left( \e_{i_{1}}^T\Omega +
        k^{-1}\omega_{i_2i_1}^{-1}\e_{i_2}^T\Omega + R(i,i') \right),
    \end{equation}
    where
    $c(i,i') =  k^{t}\omega_{i'i_{t}} \cdots
    \omega_{i_{2}i_1}$ and $R(i,i')$ is a sum of expressions of the
    form
    \begin{equation}\label{eq:fraction_term}
      \frac{ \prod k\omega_{ij}}{ \prod k\omega_{i'j'}} \e_{i''}^T\Omega
    \end{equation}
    such that the number of multiplicands in the denominator is at least two
    more than the number of multiplicands in the numerator. The trivial
    estimate yields that the supremum norm of each expression of the form
    \Cref{eq:fraction_term} is bounded from above by
    \begin{displaymath}
      \frac{||\Omega||_{\infty}^{K-2}}{k^2||\Omega||_{\min,\gamma}^{K}} ||\Omega||_{\infty}.
    \end{displaymath}

    A trivial upper bound on the number of terms in the sum $R(i,i')$
    is $2^{K}-1$, the number of nonempty subsets of
    $\{i_1,\ldots,i_{K}\}$, hence
    \begin{equation}\label{eq:R_i_i_estimate}
      ||R(i,i')||_\infty \le
      (2^{K}-1)\frac{||\Omega||_{\infty}^{K-1}}{k^2||\Omega||_{\min,\gamma}^{K}}.
    \end{equation}
    We remark that \Cref{eq:H_i_i_formula,eq:R_i_i_estimate} hold also
    when $ii'$ is a non-edge in $\bG(\Omega)$ provided $c(i,i')$ and
    $R(i,i')$ are defined to be zero.

    By substituting \Cref{eq:H_i_i_formula} into \Cref{eq:i_primeth_row} and
    rescaling each side we obtain the equation
    \begin{displaymath}
      k\frac{\e_{i'}^T \Omega M_{\gamma,k}}{\sum_{i=1}^{n} c(i,i')} =
      k\e_{i_{1}}^T\Omega +
      \omega_{i_2i_1}^{-1}\e_{i_2}^T\Omega + k\left(\frac{\sum_{i=1}^{n}
        c(i,i')R(i,i')}{\sum_{i=1}^{n} c(i,i')} +
      \frac{\e_{i'}^T\Omega}{\sum_{i=1}^{n} c(i,i')}\right).
    \end{displaymath}

    Note that the expression is the $i'$th row of $\Omega M_{\gamma,k}$,
    renormalized. The first term inside the parentheses is a convex combination
    of the $R(i,i')$, so the same bound holds for it as in
    \Cref{eq:R_i_i_estimate}. To obtain a bound for the second term, note that
    $c(i,i') \ge k^2||\Omega||^2_{\min,\gamma}$ whenever $ii' \in \bG(\Omega)$,
    because the number of multiplicands $\omega_{ij}$ in the definition of
    $c(i,i')$ is always at least two. Hence
    \begin{displaymath}
      \left\Vert\frac{\e_{i'}^T\Omega}{\sum_{i=1}^{n}
          c(i,i')}\right\Vert_\infty \le
      \frac{||\Omega||_\infty}{k^2||\Omega||^2_{\min,\gamma}} \le \frac{||\Omega||_{\infty}^{K-1}}{k^2||\Omega||_{\min,\gamma}^{K}}.
    \end{displaymath}
    By combining the estimates for the two terms inside the
    parentheses, we obtain the desired inequality.
  \end{proof}

  The inequality \Cref{eq:row-estimate} immediately follows from
  \Cref{step:eigenvector-in-cone,step:row-estimate}. This completes the proof
  of \Cref{prop:eigenvector_asymptotics}.
\end{proof}

The intuition behind \Cref{prop:eigenvector_asymptotics} is the
following. The cone $C_k$ generated by the rows of
$\Omega M_{\gamma,k}$ is the image of the cone $C$ generated by the
rows of $\Omega$ under the right action of $Q_{i_{K}}^{k}$, \ldots,
$Q_{i_1}^{k}$ in this order. Where $C$ ends up after these actions is
determined for the most part by the last action. The right action of
$Q_{i_1}^{k}$ is trivial on all standard basis vectors except one
which is translated by $k\e_{i_1}^T\Omega$. Thus we can think of the
right action of $Q_{i_1}^{k}$ as a map sending the cone $C$ towards
the direction of $\e_{i_1}^T\Omega$. This is why $\e_{i_1}^T\Omega$
appears in \Cref{eq:limit_of_difference} and why it appears with a
large weight $k$.

The second to last action is the action of $Q_{i_2}^{k}$,
sending the $C$ towards the direction of $\e_{i_2}^T\Omega$.
This action has a smaller effect than the last action, but a more
significant one than all the actions before. So $\e_{i_2}^T\Omega$
still appears in \Cref{eq:limit_of_difference}, but with a smaller
weight than $\e_{i_1}^T\Omega_k$. The rest of the actions turn out
to be negligible as $k \to \infty$.

\subsection{Projections to hyperplanes}
\label{sec:projections-to-hyperplanes}

The goal of this section is to describe the projections $p_{i \gets j}$ defined
in \Cref{eq:p-i-gets-j} by matrices.

Introduce the definition
\begin{equation}\label{eq:Q_i_gets_j}
  Q_{i\gets j} = I-\omega_{ij}^{-1}T_{ji}\Omega
\end{equation}
where $T_{ji}$ is the $n\times n$ matrix whose $(j,i)$-entry is 1 and
whose other entries are zero. The matrix $Q_{i\gets j}$ is defined if
$\omega_{ij} > 0$, in other words, if $ij$ is an edge of $\bG(\Omega)$.

Note that multiplication by $T_{ji}$ on the left has the effect of
zeroing out all rows except the $i$th row and moving the $i$th row
to the $j$th row. So in words, $Q_{i\gets j}$ is calculated in
the following way. Zero out all rows of $\Omega$ except the $i$th
row, move the $i$th row to the $j$th row, and then normalize it so
that the entry on the diagonal is 1. Subtracting this matrix from
the identity matrix gives $Q_{i\gets j}$. Note that the $j$th column of
$Q_{i\gets j}$ is zero.

\begin{lemma}\label{lemma:projection}
  If $ij$ is an edge of $\bG(\Omega)$, then the left action of
  $Q_{i\gets j}$ is the projection $p_{i \gets j}$.
\end{lemma}
\begin{proof}
  Recall that $p_{i \gets j}$ is the projection to the hyperplane
  $Z_i$ in the direction of $\e_j$ where $Z_i$ is the orthogonal
  complement of the $i$th row of $\Omega$. To show that the left
  action of $Q_{i\gets j}$ is the same transformation, it suffices to
  show that $Q_{i\gets j}\e_j = 0$ and $Q_{i\gets j}\vv = \vv$
  whenever $\e_i^T\Omega \vv = 0$.

  The first statement is clear, since the $j$th column of $Q_{i\gets
    j}$ is zero.

  For the second statement, note that the equation $Q_{i\gets j}\vv = \vv$ is
  equivalent to $T_{ji}\Omega\vv = 0$, which is in turn equivalent to saying
  that the $i$th coordinate of the vector $\Omega\vv$ is zero. But this is
  equivalent to $\e_i^T\Omega \vv = 0$, so we are done.
\end{proof}

\subsection{Convergence of linear maps}
\label{sec:convergence_of_linear_maps}

In this section, we show that the matrices $M_{\gamma,k}$ in
\Cref{theorem:convergence} act on certain codimension 1 subspaces by left
multiplication and these left actions asymptotically stabilize as $k$ goes to
infinity. We will use this fact to prove \Cref{theorem:convergence} at the end
of the section.

\begin{proposition}\label{prop:convergence-of-maps}
  Let $\Omega$ be the intersection matrix of a collection of curves
  satisfying the hypotheses of Penner's construction. Let
  $\gamma = (i_1\ldots i_Ki_1)$ be a closed path in $\bG(\Omega)$
  visiting each vertex at least once and let
  \begin{displaymath}
    M_{\gamma,k} = Q_{i_K}^{k} \cdots Q_{i_1}^{k}.
  \end{displaymath}
  Let $\ww_k^T$ be a left Perron--Frobenius eigenvector of
  $M_{\gamma,k}$. If a sequence of vectors $\vv_k$ converges to some
  vector $\vv^*$ and $\ww_k^T \vv_k = 0$ for all $k$, then
  \begin{displaymath}
    \lim_{k\to \infty} M_{\gamma,k}\vv_k = Q_{i_1\gets i_K}\cdots Q_{i_3\gets i_2}Q_{i_2\gets i_1}\vv^*.
  \end{displaymath}
\end{proposition}
\begin{proof}
  We show the following convergences by induction:
  \begin{align}
    \lim_{k\to \infty} Q_{i_1}^{k}\vv_k &= Q_{i_2 \gets i_1}\vv^* \label{eq:first_limit}\\
    \lim_{k\to \infty}
    Q_{i_2}^{k}Q_{i_1}^{k}\vv_k &= Q_{i_3 \gets i_2}Q_{i_2 \gets i_1}\vv^* \label{eq:second_limit}\\
    & \mathrel{\makebox[\widthof{=}]{\vdots}} \nonumber \\
    \lim_{k\to \infty} Q_{i_K}^{k} \cdots Q_{i_1}^{k}\vv_k
      &= Q_{i_1 \gets i_K} \cdots Q_{i_2 \gets i_1}\vv^* \label{eq:last_limit}
  \end{align}
  First we prove the base case \Cref{eq:first_limit}. Since
  \begin{displaymath}
    \lim_{k\to \infty}Q_{i_2 \gets i_1}\vv_k = Q_{i_2 \gets i_1}\vv^*,
  \end{displaymath}
  it suffices to show that
  \begin{displaymath}
    \lim_{k\to \infty} \left[Q_{i_1}^{k}\vv_k - Q_{i_2
      \gets i_1}\vv_k\right] = 0.
  \end{displaymath}

Suppose that the
  Perron--Frobenius eigenvectors $\ww_k^T$ are chosen as guaranteed by
  \Cref{prop:eigenvector_asymptotics}. Then we have
  \begin{displaymath}
    \lim_{k\to \infty}\left\lvert \left(k\e_{i_1}^T\Omega +
        \omega_{i_2i_1}^{-1}\e_{i_2}^T\Omega\right)\vv_k\right\rvert
    = \lim_{k\to \infty}\left\lvert \left(\ww_k^T -
        \left(k\e_{i_1}^T\Omega +
          \omega_{i_2i_1}^{-1}\e_{i_2}^T\Omega\right)\right)\vv_k\right\rvert
    = 0
  \end{displaymath}
  by \Cref{eq:limit_of_difference}. Therefore
  \begin{displaymath}
     \lim_{k\to \infty}\left(Q_{i_1}^{k} - Q_{i_2\gets i_1}\right)\vv_k
    = \lim_{k\to \infty}\left(kD_{i_1} \Omega +
        \omega_{i_2i_1}^{-1}T_{i_1i_2} \Omega\right)\vv_k= 0,
  \end{displaymath}
  since
  $kD_{i_1} \Omega + \omega_{i_2i_1}^{-1}T_{i_1i_2} \Omega$
  is a matrix whose $i_1$st row equals
  $k\e_{i_1}^T\Omega + \omega_{i_2i_1}^{-1}\e_{i_2}^T\Omega$
  and whose other rows are zero. This completes the proof of the base case
  \Cref{eq:first_limit}.

  Next, we describe how to obtain \Cref{eq:second_limit} from
  \Cref{eq:first_limit}. The remaining inductive steps are analogous.

  Let $M_{\gamma',k}$ be the cyclic permutation
  $Q_{i_1}^{k}Q_{i_K}^{k}\cdots Q_{i_2}^{k}$ of the product
  $M_{\gamma,k}$. If $\ww_k^T$ is a left
  Perron--Frobenius eigenvector of $M_{\gamma,k}$, then
  \begin{displaymath}
    \uu_k^T = \ww_{k}^TQ_{i_K}^{k}\cdots Q_{i_2}^{k}
  \end{displaymath}
  is a left Perron--Frobenius eigenvector of $M_{\gamma',k}$. On the
  other hand we have
  \begin{displaymath}
    \uu_k^T Q_{i_1}^{k}\vv_k = \ww_k^T M_{\gamma,k} \vv_k = \lambda_k
    \ww_k^T \vv_k = 0.
  \end{displaymath}
  In words, $Q_{i_1}^{k}\vv_k$ is orthogonal to the left
  Perron--Frobenius eigenspace of $M_{\gamma',k}$. Therefore we can
  apply \Cref{eq:first_limit} for $Q_{i_2}^k$ instead of $Q_{i_1}^{k}$
  and for $Q_{i_1}^{k}\vv_k$ instead of $\vv_k$ to obtain
  \Cref{eq:second_limit}.
\end{proof}

\newcommand\bb{{\bf b}}
\begin{proof}[Proof of \Cref{theorem:convergence}]
  Let $u_k(x) = (x-\lambda_k) v_k(x)$ and denote by $W_k$ the
  orthogonal complement of the left Perron--Frobenius eigenspace of
  $M_{\gamma,k}$. Note that $W_k$ is invariant under the left action
  of $M_{\gamma,k}$ and the induced linear transformation has
  characteristic polynomial $v_k(x)$.

  Let
  \begin{displaymath}
    \calB^* = \{\bb_1^*,\ldots,\bb_{n-1}^*\}
  \end{displaymath}
  be a basis for $Z_{i_1}$,
  the orthogonal complement of the $i_1$st row of $\Omega$. For each
  $k$, choose a basis
  \begin{displaymath}
    \calB_k = \{\bb_1^k,\ldots,\bb_{n-1}^k\}
  \end{displaymath}
  for $W_k$ such that $\bb_i^k \to\bb_i^*$ for
  $i = 1, \ldots, n-1$. This is possible, since the subspaces $W_k$
  converge to $Z_{i_1}$ by \Cref{prop:eigenvector_asymptotics}.

  By \Cref{lemma:projection}, the left action of
  $Q_{i_1\gets i_K}\cdots Q_{i_3\gets i_2}Q_{i_2\gets i_1}$, restricted to the
  hyperplane $Z_{i_1}$ equals $f_{\gamma}$. Let $A^*$ be the matrix describing
  this left action in the basis $\calB^*$. Let $A_k$ be the matrix describing the
  left action of $M_{\gamma,k}$ on $W_k$ in the basis $\calB_k$.

  By \Cref{prop:convergence-of-maps}, we have
  \begin{displaymath}
    \lim_{k\to \infty} M_{\gamma,k}\bb_i^k = Q_{i_1\gets i_K}\cdots Q_{i_3\gets
      i_2}Q_{i_2\gets i_1} \bb_i^*
  \end{displaymath}
  for $i = 1, \ldots, n-1$, therefore $A_k \to A$ and
  $\chi(A_k) \to \chi(A)$. Since $\chi(A_k) = u_k(x)$ and $\chi(A) = v(x)$,
  this completes the proof.
\end{proof}

\section{Homotopy invariance}\label{sec:homotopy-invariance}

The goal of this section is to show that the eigenvalues of the linear
transformation $f_\gamma$ defined in \Cref{eq:f-gamma} are invariant under
homotopy of $\gamma$. As a result, we will be able to determine the eigenvalues
of $f_\gamma$ for any contractible $\gamma$. It turns out that the only
eigenvalues in these cases are 0 and 1.

We say that the closed paths $\gamma$ and $\gamma'$ are homotopic in
$\bG(\Omega)$ if the naturally associated maps from $S^1$ to
$\bG(\Omega)$ are homotopic. It is easy to see that two closed paths
are homotopic if and only if they are connected by a sequence of the
insertions and removals of backtrackings (paths of the form $(iji)$)
and cyclic permutations of the vertices. Since $f_\gamma$ is not
defined for a closed path of length 0, we require that all closed
paths appearing in such a sequence have length at least two.

\begin{proposition}\label{prop:homotopy}
  If $\gamma$ and $\gamma'$ are homotopic closed paths in
  $\bG(\Omega)$, then the characteristic polynomials of $f_\gamma$ and
  $f_{\gamma'}$ are equal.
\end{proposition}
\begin{proof}
  First, we will show that removing or inserting a backtracking to $\gamma$
  does not change $f_\gamma$ as long as the last edge of $\gamma$ is not
  changed. It is easy to see that it is necessary to require at least
  that the starting vertex of $\gamma$ is fixed, since the domain of
  $f_\gamma$ depends on the first vertex of $\gamma$.

  Since inserting a backtracking is the inverse operation of removing one, it
  suffices to show the statement for removals. For this, let
  $\gamma = (i_1 \ldots i_K i_1)$, let $2\le k\le K-1$ and suppose that
  $i_{k-1} = i_{k+1} = i$. We will to show that the removal of the backtracking
  $(i_{k-1} i_k i_{k+1})$ leaves $f_\gamma$ unchanged.

  If $k \ge 3$, then the composition
  \begin{displaymath}
    p_{i_{k+2} \gets i} \circ p_{i \gets i_k} \circ p_{i_{k} \gets i} \circ p_{{i} \gets i_{k-2}}
  \end{displaymath}
  appears in the formula \Cref{eq:f-gamma}. We are using the convention that
  $i_{K+1} = i_1$. However, this is the same as
  $p_{i_{k+2} \gets i} \circ p_{i \gets i_{k-2}}$ for the following reasons.
  The image of $p_{i \gets i_{k-2}}$ is the hyperplane $Z_i$, the orthogonal
  complement of the $i$th row of $\Omega$. The subsequent projection,
  $p_{i_{k} \gets i}$ is in the direction of $\e_i$. However, we have
  $\e_i \in Z_i$, because the diagonal entries of $\Omega$ are zero, so
  $\e_i^T \Omega \e_i = 0$. Hence the image of
  $p_{i_{k} \gets i} \circ p_{i \gets i_{k-2}}$ is still inside $Z_i$. As a
  consequence, the next projection, $p_{i \gets i_k}$, which also projects
  onto $Z_i$, does not have any effect. This shows that
  \begin{displaymath}
    p_{i_{k+2} \gets i} \circ p_{i \gets i_k} \circ p_{i_{k} \gets i}
    \circ p_{{i} \gets i_{k-2}} = p_{i_{k+2} \gets i} \circ p_{i_{k} \gets i}
    \circ p_{{i} \gets i_{k-2}}.
  \end{displaymath}
  But now the subsequent projections $p_{i_{k+2} \gets i}$ and
  $p_{i_{k} \gets i}$ are both projections in the direction of $\e_{i}$, hence
  $p_{i_{k+2} \gets i} \circ p_{i_{k} \gets i} = p_{i_{k+2} \gets i}$. So we
  indeed end up with shorter the composition
  $p_{i_{k+2} \gets i} \circ p_{i \gets i_{k-2}}$.

  If $k = 2$, this argument needs to be slightly modified. We now have
  \begin{displaymath}
    f_\gamma = (p_{i_1 \gets i_K} \circ  \cdots \circ p_{i_4 \gets i}
    \circ p_{i \gets i_2} \circ p_{i_2 \gets i})|_{Z_{i}}.
  \end{displaymath}
  The image of $Z_i$ under $p_{i_2 \gets i}$ is contained in $Z_i$, since
  $\e_i \in Z_i$. Now the same arguments as above show that $p_{i \gets i_2}$
  and then $p_{i_2 \gets i}$ can be omitted from the composition. This
  completes the proof of the fact that $f_\gamma$ is invariant under homotopy
  of $\gamma$ rel the last edge of $\gamma$.

  Another fact we need is that cyclic permutation of the vertices of $\gamma$
  does not change the characteristic polynomial of $f_\gamma$. One can see this
  directly from the formula \Cref{eq:f-gamma}, because cyclic permutation
  of the vertices of $\gamma$ changes $f_\gamma$ by conjugation.

  We can now give the proof of the proposition. Suppose that $\gamma$ and
  $\gamma'$ are homotopic and therefore connected by a sequence of insertions
  and removals of backtrackings and cyclic permutations of the vertices. Since
  our paths have length at least 2, we can always permute the vertices before
  an insertion of removal of a backtracking so that the last edge is unchanged.
  So in each step $f_\gamma$ either does not change or it changes by
  conjugation. Either way, the characteristic polynomial does not change.
\end{proof}

\begin{corollary}\label{cor:contractible-almost-identity}
  If the matrix $\Omega$ has size $n\times n$ and the path $\gamma$ is contractible,
  then the characteristic polynomial of $f_\gamma$ takes the form
  \begin{displaymath}
    \chi(f_\gamma) = x(x-1)^{n-2}.
  \end{displaymath}
\end{corollary}
\begin{proof}
  By \Cref{prop:homotopy}, the characteristic polynomial of $f_\gamma$
  is the same as that of $f_{\gamma'}$ where $\gamma' = (iji)$ is a path of
  length two. Writing out the definition of $f_{\gamma'}$, we have
  \begin{displaymath}
    f_{(iji)} = (p_{i \gets j} \circ p_{j\gets i})|_{Z_{i}}.
  \end{displaymath}
  As we have observed in the proof of \Cref{prop:homotopy}, we have
  $\e_i \in Z_i$, so the image of $Z_i$ under $p_{j\gets i}$ is contained in
  $Z_i$. Therefore $p_{i \gets j}$ can be omitted from the formula and
  $f_{(iji)}$ is just the projection $p_{j\gets i}|_{Z_{i}}$ projecting $Z_i$
  to a codimension 1 subspace of $Z_i$. Hence $\chi(f_\gamma) = x(x-1)^{n-2}$
  as stated.
\end{proof}

\section{A simple criterion for realizing degrees}
\label{sec:rank-criterion}

In this section, we give a simple way to certify that a given degree can be
realized on a given surface.


We will need the following fact relating the rank of the intersection matrix
$\Omega$ to the 1-eigenspaces of products of the matrices $Q_i$.

\begin{proposition}\label{prop:multiplicity-of-one}
  Let $\psi$ be a pseudo-Anosov map arising from Penner's construction using a
  collection of curves with an $n \times n$ intersection matrix $\Omega$. If
  $M$ is the product of the $Q_i$ describing $\psi$, then 1 is an eigenvalue of
  $M$ with multiplicity $n-r$ where $r = \rank(\Omega)$. In particular, the
  characteristic polynomial of $M$ takes the form $(x-1)^{n-r}p(x)$ where
  $\deg(p) = r$ and $p(1) \ne 0$.
\end{proposition}
\begin{proof}
  First we show that the multiplicity is at least $n-r$.
  This is because the left action of every $Q_i = I + D_i\Omega$ is the
  identity on the null space $\Nul(\Omega)$ of $\Omega$, so the left action
  of $M$ on $\Nul(\Omega)$ is also the identity. Moreover, the
  dimension of $\Nul(\Omega)$ is $n-r$.

  Now we turn to showing that the multiplicity is at most
  $n-r$. Since $M$ acts on $\Nul(\Omega)$ as the identity,
  it has a well-defined left action
  \begin{displaymath}
    \ell_M: \hV \to \hV
  \end{displaymath}
  on the quotient space $\hV = \RR^n/\Nul(\Omega)$. It suffices to show that 1
  is not an eigenvalue of $\ell_M$.

  In other to show this, we will consider the quadratic form $h$ on $\RR^n$
  defined by the formula
  \begin{displaymath}
    h(\vv) = \frac12 \vv^T \Omega \vv.
  \end{displaymath}
  The function $h$ can be thought of as a height function on $\RR^n$. It was
  shown in Proposition 2.1 of \cite{ShinStrenner15} that
  \begin{equation}\label{eq:height-increase}
    h(Q_i\vv) - h(\vv) = ||Q_i\vv - \vv||^2
  \end{equation}
  for all $i = 1, \ldots, n$ and $\vv\in \RR^n$. In words, the matrices $Q_i$ act on
  $\RR^n$ by not decreasing the height. Moreover, there is no increase in the
  height if and only if $\vv$ is fixed by $Q_i$.

  Now suppose that $\ell_M(\hv) = \hv$ for some $\hv \in \hV$; we want to show
  that $\hv = 0$. Then there is some $\vv \in \RR^n$ such that
  $M\vv - \vv \in \Nul(\Omega)$, so $\Omega M \vv = \Omega \vv$. Hence
  \begin{displaymath}
    h(M\vv) = \frac12 \vv^TM^T \Omega M \vv = \frac12 \vv^T \Omega \vv = h(\vv).
  \end{displaymath}
  By \Cref{eq:height-increase}, this is only possible if $Q_i\vv = \vv$ for all
  $i = 1, \ldots, n$, since each $Q_i$ appears in $M$. It follows that
  $D_i\Omega\vv = 0$ for all $i = 1, \ldots, n$. Since $\sum_{i=1}^nD_i = I$, we have
  $\Omega\vv = 0$, hence $\hv$ is the zero vector of $\hV$. Therefore
  1 is indeed not an eigenvalue of $\ell_M$.
\end{proof}

\begin{corollary}\label{prop:rank-bounds-degree}
  Let $\lambda$ be a pseudo-Anosov stretch factor arising from Penner's
  construction using a collection of curves with intersection
  matrix $\Omega$. Then $\deg(\lambda) \le \rank(\Omega)$.
\end{corollary}
\begin{proof}
  The number $\lambda$ is an eigenvalue of a matrix $M$ and hence a root of the
  polynomial $p(x)$ in \Cref{prop:multiplicity-of-one}.
\end{proof}

The following theorem gives a recipe for constructing a stretch factor with a
specified algebraic degree.

\begin{theorem}\label{theorem:construction-of-degree-r}
  Let $\Omega$ be the intersection matrix of a collection of curves $C$
  satisfying the hypotheses of Penner's construction such that
  $\rank(\Omega) = r$. Let $\gamma = (i_1\ldots i_Ki_1)$ be a contractible
  closed path in $\bG(\Omega)$ visiting each vertex at least once and let
  \begin{displaymath}
    M_{\gamma,k} = Q_{i_K}^{k} \cdots Q_{i_1}^{k}.
  \end{displaymath}
  Let $f_{k}$ be the pseudo-Anosov mapping class described by the matrix
  $M_{\gamma,k}$ and let $\lambda_k$ be its stretch factor. Then
  $\deg(\lambda_k) = r$ for all but finitely many $k$.
\end{theorem}
\begin{proof}
  Let $u_k(x)$ be the characteristic polynomial of $M_{\gamma,k}$.
  \Cref{prop:multiplicity-of-one} shows that $u_k(x) = (x-1)^{n-r}p_k(x)$ where
  $p_k(1) \ne 0$ and the degree of $p_k(x)$ is $r$. Since $\lambda_k$ is a root
  of $p_k(x)$, it suffices to show that $p_k(x)$ irreducible if $k$ is large
  enough.



  By \Cref{theorem:convergence} and
  \Cref{cor:contractible-almost-identity}, we have
  \begin{displaymath}
    \lim_{k\to \infty}\frac{u_k(x)}{x-\lambda_k} = x(x-1)^{n-2}.
  \end{displaymath}

  So all roots of $p_k(x)$ except for $\lambda_k$ converge to either 0 or 1.
  Note that $p_k(0) \ne 0$, since the matrices $Q_i$ are invertible, hence so
  is $M_{\gamma,k}$. Therefore all roots are different from their limits for
  all $k$ and \Cref{lemma:galois-conjugates} implies that $p_k(x)$ is
  indeed irreducible if $k$ is large enough.
\end{proof}

As an immediate corollary of \Cref{theorem:construction-of-degree-r} and
\Cref{prop:train-track-orientable-foliation}, we have the following.

\begin{theorem}\label{theorem:curves-and-degree}
  If the surface $S$ admits a filling collection of curves $C$ with
  inconsistent markings such that $\rank(i(C,C)) = r$, then $r \in D(S)$. In
  addition, if $C$ is completely left-to-right or completely right-to-left,
  then $r \in D^+(S)$.
\end{theorem}

\section{Collections of curves}
\label{sec:collections_of_curves}

In this section, we construct filling collections of curves on various surfaces.
By \Cref{theorem:curves-and-degree}, the integers that arise are ranks of the
intersection matrices also arise as algebraic degrees of stretch factors.

We consider both orientable and nonorientable surfaces. In the orientable case, the
constructions are fairly straightforward. The nonorientable case is also not
difficult, but we will need to do more case-by-case analysis for surfaces with
small Euler characteristic.

\subsection{Orientable surfaces}

Recall the definition of completely left-to-right and completely right-to-left
collections of curves from \Cref{sec:oriented_collections}. Let $\overline{S}$
denote the closed surface obtained from $S$ by filling in the punctures.

\begin{proposition}\label{prop:examples_of_filling_pairs}
  Let $S$ be an orientable surface. For all
  $1\le r \le \frac12 \dim(\Teich(S))$ there is a filling collection
  of curves $C$ with inconsistent markings on $S$ such that
  $\rank(i(C,C)) = 2r$.

  Moreover, for $1\le r \le \frac12 \dim(H_1(\overline{S}))$, the collection
  $C$ can be chosen to be completely left-to-right or completely right-to-left.
\end{proposition}

\begin{proof}
  Since $S$ is orientable, $C$ is necessarily a union of two
  multicurves $A$ and $B$, where the curves of $A$ are marked
  consistently with the orientation of $S$, and the curves of $B$
  are marked inconsistently with the orientation of $S$. Note that
  $\rank(i(C,C)) = 2\rank(i(A,B))$.

  Let $S_{g,n}$ be the orientable surface of genus $g$ with $n$
  punctures. In the special case $(g,n)=(4,3)$,
  \Cref{fig:maximal_filling} shows a pair of filling multicurves $A$
  and $B$ on $S_{g,n}$ with $\frac12 \dim(\Teich(S)) = 3g-3+n$ simple
  closed curves in each multicurve. This construction generalizes for
  all $S_{g,n}$ (where $g\ge 2$) in the following way. The separating
  curves of $B$ shown on \Cref{fig:maximal_filling} divide $S_{4,3}$
  to two once-punctured tori on the left and right, and two
  twice-punctured tori in the middle. To draw the analogous picture
  for $S_{g,3}$, change the number of twice-punctured tori in the
  middle from two to $g-2$. Then, to get the curve systems on
  $S_{g,n}$ for arbitrary $n$, change the number of punctures in the
  once-punctured torus on the right, and change the number of parallel
  curves of $A$ and $B$ around the punctures accordingly.

  \begin{figure}[ht]
    \labellist
    \small\hair 2pt
    \pinlabel {$A$} at 13 45
    \pinlabel {$B$} at 45 52
    \endlabellist
    \centering
    \includegraphics[width=13cm]{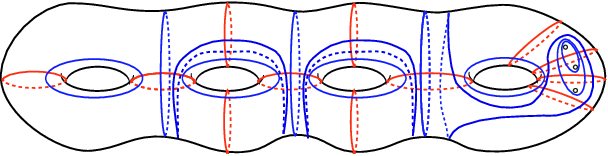}
    \caption{A maximal pair of filling multicurves on $S_{4,3}$.}
    \label{fig:maximal_filling}
  \end{figure}

  By numbering the curves in each multicurve left-to-right and
  top-to-bottom, $i(A,B)$ takes the form
  \begin{equation}\label{eq:max_rank_matrix}
    \left(
    \begin{array}{cccccccccccc}
      1 & 0 & 0 & 0 & 0 & 0 & 0 & 0 & 0 & 0 & 0 & 0\\
      1 & 2 & 2 & 1 & 0 & 0 & 0 & 0 & 0 & 0 & 0 & 0\\
      0 & 0 & 2 & 1 & 0 & 0 & 0 & 0 & 0 & 0 & 0 & 0\\
      0 & 0 & 0 & 1 & 0 & 0 & 0 & 0 & 0 & 0 & 0 & 0\\
      0 & 0 & 2 & 1 & 2 & 2 & 1 & 0 & 0 & 0 & 0 & 0\\
      0 & 0 & 0 & 0 & 0 & 2 & 1 & 0 & 0 & 0 & 0 & 0\\
      0 & 0 & 0 & 0 & 0 & 0 & 1 & 0 & 0 & 0 & 0 & 0\\
      0 & 0 & 0 & 0 & 0 & 2 & 1 & 2 & 2 & 1 & 0 & 0\\
      0 & 0 & 0 & 0 & 0 & 0 & 0 & 0 & 0 & 1 & 0 & 0\\
      0 & 0 & 0 & 0 & 0 & 0 & 0 & 0 & 2 & 1 & 2 & 2\\
      0 & 0 & 0 & 0 & 0 & 0 & 0 & 0 & 2 & 1 & 2 & 0\\
      0 & 0 & 0 & 0 & 0 & 0 & 0 & 0 & 2 & 1 & 0 & 0\\
    \end{array}\right).
  \end{equation}

  From the pattern, it is not hard to see that $i(A,B)$ has nonzero
  determinant for all $S_{g,n}$ where $g\ge 2$.

  Note that $A$ and the multicurve consisting of the $g$ curves of $B$
  around the holes still fill the surface. Therefore $A$ and any
  submulticurve of $B$ that contains those $g$ curves also fill. This
  gives examples for pairs of filling multicurves with intersection
  matrices of rank $r$ for $g\le r \le 3g-3+n$.
  \begin{figure}[ht]
    \centering
    \includegraphics[width=13cm]{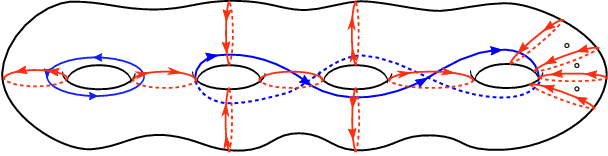}
    \caption{Pairs of multicurves realizing ranks $1\le r \le g$.}
    \label{fig:figure_eight_filling}
  \end{figure}

  To obtain examples for all ranks $1\le r\le g-1$, link together the
  $g$ curves around the holes one by one as on
  \Cref{fig:figure_eight_filling}, resulting in multicurves $B'$
  consisting of fewer and fewer curves. These multicurves still fill
  with $A$, and the columns of $i(A,B')$ are linearly independent,
  because for every curve in $B'$ there is a curve in $A$ that
  intersects only that curve. Note that when $1\le r \le g$, the pairs of
  multicurves in the examples shown on \Cref{fig:figure_eight_filling}
  are completely left-to-right if oriented as on the figure, because
  the red curves always cross the blue curves from left to right (cf.
  \Cref{fig:crossings}). This completes the case $g\ge 2$.

  \begin{figure}[ht]
    \labellist
    \small\hair 2pt
    \pinlabel {$A$} at 105 20
    \pinlabel {$B$} at 70 74
    \endlabellist
    \centering
    \includegraphics[width=6cm]{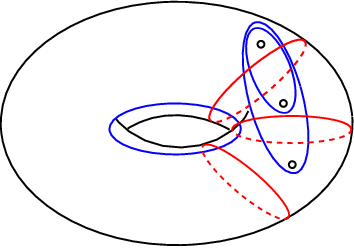}
    \caption{The case $g=1$.}
    \label{fig:torus_case}
  \end{figure}

  \begin{figure}[ht]
    \labellist
    \small\hair 2pt
    \pinlabel {$A$} at 13 45
    \pinlabel {$B$} at 152 45
    \endlabellist
    \centering
    \includegraphics[width=6cm]{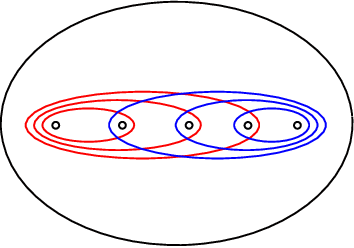}
    \caption{The case $g=0$.}
    \label{fig:disk_case}
  \end{figure}

  \Cref{fig:torus_case,fig:disk_case} show examples with
  $\rank(i(A,B)) = 3g-3+n$ in the cases $g=1,n\ge 1$ and $g=0,n\ge 4$,
  respectively. In all these cases, there is a curve in $B$ that
  intersects all curves of $A$ and which alone fills the surface with
  $A$. Hence once again we can drop curves from $B$ preserving the
  filling property and decreasing the rank. The rank 1 example for
  $g=1$ is completely left-to-right (with the appropriate markings),
  so this proves the second part of the proposition when $g=1$, while
  for $g=0$ there is nothing to prove since $\dim(H_1(\overline{S_{0,n}}))) =
  0$ for all $n$.

  In the remaining cases $g=1$, $n=0$ and $g=0$, $n<4$, the formula $\dim(\Teich(S_{g,n})) = 6g-6+2n$
  does not hold. The case $(g,n) = (1,0)$ is straightforward to check as we
  have $\dim(\Teich(S_{1})) = 2$. In the cases $g=0$, $n<4$, we have
  $\dim(\Teich(S_{g,n})) = 0$, so there is nothing to check.
\end{proof}

\subsection{Nonorientable surfaces}
\label{sec:nonor_surfaces}
Let $N_{g,n}$ be the nonorientable surface of genus $g$---the connected sum of
$g$ copies of the projective plane---with $n$ punctures. We also use the term
\emph{crosscap} for the projective plane. Analogously to the orientable case,
we abbreviate $N_{g,0}$ as $N_{g}$.

One way to obtain a nonorientable surface is to cut an open disk out
of a surface and glue the resulting boundary component to the boundary
of a M\"obius strip. In other words, the resulting surface is the connected sum of the
original surface and a crosscap. We refer to this process as
\emph{attaching a crosscap}. On figures, it is common to
mark the location of the above surgery by a cross inside a disk. For
example, \Cref{fig:nonor-left-to-right} on the left shows a surface
obtained from the sphere by attaching five crosscaps, while
\Cref{fig:N_5_rank_5} shows a surface obtained from $S_2$ by attaching
one crosscap.

A slightly different way of thinking about attaching a crosscap is
by cutting an open disk out and identifying antipodal points of the
resulting boundary component. That means that if a curve enters an
attached crosscap, then it exits the crosscap at the antipodal point.

On many figures in this section, we represent a nonorientable surface
as an orientable surface with crosscaps attached. We fix an
orientation on the complement of the crosscaps. We color parts of a
marked curve (cf.~\Cref{sec:penners-construction}) on such a surface
using two different colors depending on whether the embedding of the
regular neighborhood of the curve is orientation-preserving or
orientation-reversing on that part. Note that the color of the curve
changes when it goes through a crosscap.

The result of this section is the following.
\begin{proposition}\label{prop:curves_nonor}
  If $g\ge 3$ and $g+n \ge 5$ or $1 \le g \le 2$ and $g+n \ge 4$, then for all
  $3\le r \le \dim(\Teich(N_{g,n})) = 3g+2n-6$ there is a filling
  inconsistently marked collection of curves $C$ on $N_{g,n}$ such that
  $\rank(i(C,C)) = r$.

  If $(g,n) = (4,0)$ or $(3,1)$, then for all
  $3\le r \le \dim(\Teich(N_{g,n}))-1$ there is a filling
  inconsistently marked collection of curves $C$ on $N_{g,n}$ such that
  $\rank(i(C,C)) = r$.

  Moreover, when $3\le r \le \dim(H_1(N_g,\RR)) = g-1$, the collection $C$ on
  $N_{g,n}$ can be chosen to be completely left-to-right or completely
  right-to-left.
\end{proposition}
See \Cref{table:nonor-curves} for a summary of the surfaces with small Euler characteristic.
\begin{table}[ht]
  \centering
  \begin{tabular}{c|ccccc}
      n$\backslash$g & 1 & 2 & 3 & 4 & 5 \\
    \hline
      0   & $\emptyset$ & $\emptyset$ & $\emptyset$ & -1 & E \\
      1   & $\emptyset$ & $\emptyset$ & -1 & E & E \\
      2   & $\emptyset$ & E & E & E & E \\
      3   & E & E & E & E & E \\
  \end{tabular}
  \caption{The surfaces for which every number between 3 and
    $\dim(\Teich(N_{g,n}))$ are realized as $\rank(i(C,C))$ are marked
    by E. The surfaces for which all these numbers except the maximum
    $\dim(\Teich(N_{g,n}))$ are realized are marked by -1. The
    surfaces that do not admit pseudo-Anosov maps are marked by
    $\emptyset$.}
  \label{table:nonor-curves}
\end{table}

The proof is based on \Cref{prop:examples_of_filling_pairs} and two lemmas that
we discuss next.

Suppose we have a filling collection of curves $C$ on a surface. Provided $C$
satisfies certain conditions, \Cref{lemma:attaching_a_crosscap} says that
attaching a crosscap to the surface allows extending $C$ to a filling
collection on the new surface in a way that $\rank(i(C,C))$ increases
by $0,2$ or $3$. \Cref{lemma:adding-a-puncture} says that adding a
puncture allows increasing $\rank(i(C,C))$ by $0$ or $2$.

In the following statement, let $\calN$ denote a small regular open neighborhood.

\begin{lemma}
  \label{lemma:attaching_a_crosscap}
  Let $C = \{c_i\}$ be a collection of inconsistently marked simple
  closed curves on a surface $S$ which are in pairwise minimal
  position. Let $R$ be a component of the complement of $\calN(C)$. Note
  that $\partial R$ is a union of arcs $a_j$, each of which lies on the
  boundary of some $\overline{\calN(c_{i_j})}$.

  Let $a_1$ and $a_2$ be two arcs such that
  \begin{itemize}
  \item $c_{i_1}$ and $c_{i_2}$ are non-isotopic and disjoint,
  \item there exists an arc $b$ inside $R$ connecting $a_1$ and $a_2$
    such that the markings of $c_{i_1}$ and $c_{i_2}$ induce
    different orientations on
    $\calN(c_{i_1}) \cup \calN(b) \cup \calN(c_{i_2}) \homeo S_{0,3}$.
  \end{itemize}

  Let $S'$ be the surface obtained by attaching a crosscap to $S$
  inside $R$. Consider the curves $d_1$, $d_2$ and $e$ on $S'$
  illustrated on \Cref{fig:three-new-curves}. The curves $d_1$ and
  $d_2$ are obtained from $c_{i_1}$ and $c_{i_2}$ by replacing the
  arcs $a_1$ and $a_2$ with arcs going around the crosscap in $R$. The
  curve $e$ is obtained from $c_{i_1}$ and $c_{i_2}$ by replacing the
  arcs $a_1$ and $a_2$ by two arcs going through the crosscap in $R$.
  Note that $e$ is simple, since $c_1$ and $c_2$ are disjoint. We endow
  $d_1,d_2$ and $e$ with markings illustrated on \Cref{fig:three-new-curves}.

  Finally, we introduce notations for the following three collections of curves on $S'$:
  \begin{align*}
    C' &= C \cup \{e\} \\
    C'' &= C \cup \{e,d_1\}\\
    C''' &= C \cup \{e,d_1,d_2\}.
  \end{align*}

  The statement of the lemma is that the following hold for $C'$, $C''$ and $C''$:
  \begin{enumerate}[(i)]
  \item \label{item:inconsistent} $C'$, $C''$ and $C'''$ are
    inconsistently marked;
  \item\label{item:minimal-pos} $C'$, $C''$ and $C'''$ are in
    pairwise minimal position;
  \item\label{item:ranks}
    $\rank(i(C,C)) = \rank(i(C',C')) = \rank(i(C'',C'')) - 2 =
    \rank(i(C''',C''')) - 3$;
  \item\label{item:filling} if $C$ fills $S$, then $C'$, $C''$
    and $C'''$ fill $S'$.
  \end{enumerate}
\end{lemma}

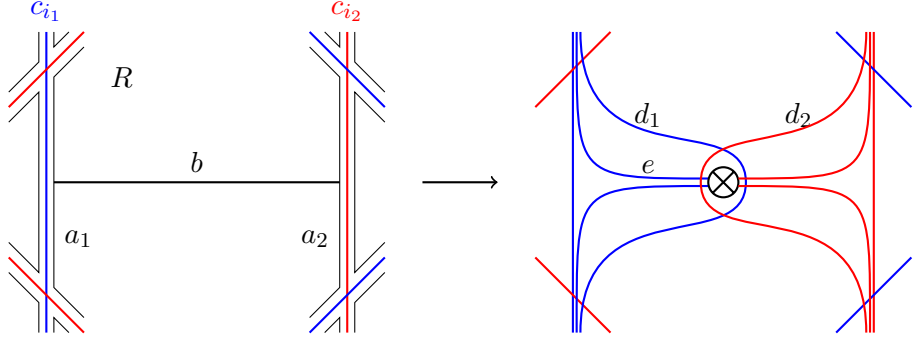
\begin{figure}
  \centering
  \begin{tikzpicture}[thick]
    \node at (1,1.4) {$R$};
    \node[right] at (0.1,-0.75) {$a_1$};
    \draw[blue] (0,-2) -- (0,2) node[above] {$c_{i_1}$};
    \draw[red] (-0.5,-1) -- (0.5,-2);
    \draw[red] (-0.5,1) -- (0.5,2);

    \draw[thin] (-0.1,-2) -- ++(0,0.4) -- ++(-0.4,0.4) ++(0,0.4) --
    ++(0.4,-0.4) -- ++(0,2.4) -- ++(-0.4,-0.4) ++(0,0.4) --
    ++(0.4,0.4) -- ++(0,0.4) ++(0.2,0) -- ++(0,-0.2) -- ++(0.2,0.2)
    ++(0.2,-0.2) -- ++(-0.4,-0.4) -- ++(0,-2.8) -- ++(0.4,-0.4)
    ++(-0.2,-0.2) -- ++(-0.2,0.2) -- ++(0,-0.2);

    \draw (0.1,0) -- node[above] {$b$} (3.9,0);
    \draw[->] (5,0) -- (6,0);
    \begin{scope}[xshift=4cm,rotate=180,]
    \node[left] at (0.1,0.75) {$a_2$};
      \draw[red] (0,-2) node[above] {$c_{i_2}$} -- (0,2);
      \draw[blue] (-0.5,-1) -- (0.5,-2);
      \draw[blue] (-0.5,1) -- (0.5,2);
    \draw[thin] (-0.1,-2) -- ++(0,0.4) -- ++(-0.4,0.4) ++(0,0.4) --
    ++(0.4,-0.4) -- ++(0,2.4) -- ++(-0.4,-0.4) ++(0,0.4) --
    ++(0.4,0.4) -- ++(0,0.4) ++(0.2,0) -- ++(0,-0.2) -- ++(0.2,0.2)
    ++(0.2,-0.2) -- ++(-0.4,-0.4) -- ++(0,-2.8) -- ++(0.4,-0.4)
    ++(-0.2,-0.2) -- ++(-0.2,0.2) -- ++(0,-0.2);
    \end{scope}

    \begin{scope}[xshift=7cm]
    \node at (1,0.9) {$d_1$};
    \node at (3,0.9) {$d_2$};
    \node at (1,0.2) {$e$};
    \draw[blue] (0,-2) -- (0,2);
    \draw[red] (-0.5,-1) -- (0.5,-2);
    \draw[red] (-0.5,1) -- (0.5,2);

    \draw[blue] (0.05,-2) .. controls (0.05,-0.05) and (0,-0.05) ..
    (1.8,-0.05);
    \draw[blue,yscale=-1] (0.05,-2) .. controls (0.05,-0.05) and (0,-0.05) .. (1.8,-0.05);
    \draw[blue] (0.1,-2) .. controls (0.1,-0.05) and (2.3,-1) ..
    (2.3,0) .. controls (2.3,1) and (0.1,0) .. (0.1,2);

    \draw (2,0) circle (0.2) ++(0.2/1.41,0.2/1.41) --
    ++(-0.4/1.41,-0.4/1.41) ++(0,+0.4/1.41) -- ++(0.4/1.41,-0.4/1.41) ;

    \begin{scope}[xshift=4cm,rotate=180,]
      \draw[red] (0,-2) -- (0,2);
      \draw[blue] (-0.5,-1) -- (0.5,-2);
      \draw[blue] (-0.5,1) -- (0.5,2);

    \draw[red] (0.05,-2) .. controls (0.05,-0.05) and (0,-0.05) ..
    (1.8,-0.05);
    \draw[red,yscale=-1] (0.05,-2) .. controls (0.05,-0.05) and (0,-0.05) .. (1.8,-0.05);
    \draw[red] (0.1,-2) .. controls (0.1,-0.05) and (2.3,-1) ..
    (2.3,0) .. controls (2.3,1) and (0.1,0) .. (0.1,2);
    \end{scope}
    \end{scope}
\end{tikzpicture}
  \caption{Creating three new curves by attaching a crosscap.}
  \label{fig:three-new-curves}
\end{figure}

\begin{proof}
  The statements (\ref{item:inconsistent}) and (\ref{item:filling}) are clear.

  Next, we describe how the statement (\ref{item:minimal-pos}) can be verified
  using the bigon criterion. The bigon criterion says that two curves are in
  minimal position if an only if they do not form a bigon: an embedded disk
  whose boundary is a union of an arc of one curve and an arc of the other and
  whose interior is disjoint from the two curves \cite[Proposition
  1.7]{FarbMargalit12}.

  First we show that $d_1$ does not form a bigon with the curves of
  $C$. (By symmetry, the same will be true for
  $d_2$.) Assume for contradiction that
  $d_1$ does form a bigon with some curve $f$ of $C$. Note that in $S'$, the curves
  $d_1$ and $c_{i_1}$ bound an annulus
  $\mathcal{A}$ with an attached crosscap. So there are two possibilities: either
  the bigon is on the side of $d_1$ contained in
  $\mathcal{A}$ or the other side. In the first scenario the bigon has to be
  entirely contained in
  $\mathcal{A}$, otherwise there would be a bigon between $f$ and
  $c_{i_1}$. But this cannot happen, since every curve of
  $C$ entering $\mathcal{A}$ through $d_1$ exits it through
  $c_{i_1}$. Using this last observation, one can see that the second scenario
  is also not plausible, since we could obtain a bigon between
  $f$ and $c_{i_1}$ in the \emph{original} surface
  $S$ by extending the bigon between $f$ and
  $d_1$ into the annulus $\mathcal{A}$ in $S$.

  Next, we show that there is no bigon between $e$ and the curves of
  $C$. For this, note that $c_{i_1}$, $c_{i_2}$ and $e$ bound a pair of pants
  $\mathcal{P}$. If there was a bigon between $e$ and some curve $f$ of $C$, then there
  are again two possibilities: the bigon is on the side of $e$ contained in $\mathcal{P}$
  or the other side. The first scenario is impossible for reasons
  similar to above: every curve of $C$ entering $\mathcal{P}$ through $e$ exits
  it either through $c_{i_1}$ or through $c_{i_2}$. In the second
  scenario, consider the arc $AB$ of $f$ that forms a bigon with $e$.
  Following $f$ from this arc in both directions into $\mathcal{P}$ until $f$
  exits $\mathcal{P}$ yields a longer arc $A'B'$ of $f$ (\Cref{fig:pants}). One
  endpoint of this longer arc is on $c_{i_1}$, the other is on
  $c_{i_2}$, otherwise $c_{i_1}$ or $c_{i_2}$ forms a bigon with $f$.
  Note that the arc $AB$ of $f$ is part of the original surface $S$ so
  it does not intersect the core curve of the attached crosscap.
  However, from the fact that $A'$ and $B'$ cannot both lie on
  $c_{i_1}$ or $c_{i_2}$, it is easy to see that the arc $AB$ of $e$
  forming the other side of the bigon goes through the crosscap and
  hence intersects its core curve once. So the arc $AB$ of $e$ and the
  arc $AB$ of $f$ cannot be homotopic rel $A$ and $B$, a
  contradiction.

  \begin{figure}[htb]
    \labellist
    \small\hair 2pt
    \pinlabel {$c_{i_1}$} [ ] at 52 132
    \pinlabel {$c_{i_2}$} [ ] at 110 131
    \pinlabel {$e$} [ ] at 145 168
    \pinlabel {$\mathcal{P}$} [ ] at 78 162
    \pinlabel {$A'$} [ ] at 40 76
    \pinlabel {$B'$} [ ] at 100 76
    \pinlabel {$A$} [ ] at 48 18
    \pinlabel {$B$} [ ] at 140 35
    \pinlabel {$f$} [ ] at 125 2
    \endlabellist
    \centering
    \includegraphics[scale=0.5]{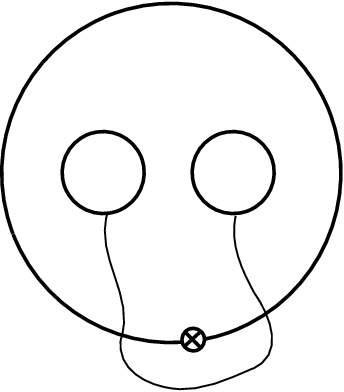}
    \caption{Ruling out the possibility of a bigon
      between $e$ and a curve $f\in C$.}
    \label{fig:pants}
  \end{figure}
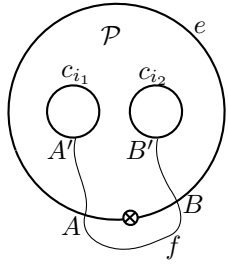

  To show that $d_1$ and $d_2$ do not form a bigon, note that they intersect
  twice, so there are only four regions that are candidates for bigons. One of
  them contains the crosscap, so it is not a bigon. Two other regions are ruled
  out because $c_{i_1}$ and $c_{i_2}$ are not nullhomotopic. The fourth region
  is a bigon if any only if $c_{i_1}$ and $c_{i_2}$ are isotopic in $S$, but
  our assumption is that they are not isotopic. Hence $d_1$ and $d_2$ are in
  minimal position.

  Finally, the pairs $(d_1,e)$ and $(d_2,e)$ are
  checked similarly, again using the fact that $c_{i_1}$ and $c_{i_2}$
  are not nullhomotopic and not isotopic. This finishes the proof of (\ref{item:minimal-pos}).

  Let $C_0 = C - \{c_{i_1},c_{i_2}\}$. We can write $i(C''',C''')$ in
  the following block form.
  \begin{multline*}
    i(C''',C''') =  \\
    =\begin{pmatrix}
      0 & i(C_0,c_{i_1}) & i(C_0,c_{i_2}) & i(C_0,d_1) & i(C_0,d_2) &
      i(C_0,e) \\
      i(c_{i_1},C_0) & 0 & i(c_{i_1},c_{i_2}) & i(c_{i_1},d_1) & i(c_{i_1},d_2) &
      i(c_{i_1},e) \\
      i(c_{i_2},C_0)  & i(c_{i_2},c_{i_1}) & 0 & i(c_{i_2},d_1) & i(c_{i_2},d_2) &
      i(c_{i_2},e) \\
      i(d_1,C_0)  & i(d_1,c_{i_1}) & i(d_1,c_{i_2}) & 0 & i(d_1,d_2) &
      i(d_1,e) \\
      i(d_2,C_0)  & i(d_2,c_{i_1}) & i(d_2,c_{i_2}) & i(d_2,d_1) & 0 &
      i(d_2,e) \\
      i(e,C_0)  & i(e,c_{i_1}) & i(e,c_{i_2}) & i(e,d_1) & i(e,d_2) &
      0 \\
    \end{pmatrix} = \\
    =\begin{pmatrix}
      0 & X & Y & X & Y & X + Y \\
      X^T & 0 & 0 & 0 & 0 & 0 \\
      Y^T  & 0 & 0 & 0 & 0 & 0 \\
      X^T  & 0 & 0 & 0 & 2 & 2 \\
      Y^T  & 0 & 0 & 2 & 0 & 2 \\
      X^T + Y^T  & 0 & 0 & 2 & 2 & 0 \\
    \end{pmatrix} \sim
    \begin{pmatrix}
      0 & X & Y & 0 & 0 & 0 \\
      X^T & 0 & 0 & 0 & 0 & 0 \\
      Y^T  & 0 & 0 & 0 & 0 & 0 \\
      0 & 0 & 0 & 0 & 2 & 2 \\
      0  & 0 & 0 & 2 & 0 & 2 \\
      0  & 0 & 0 & 2 & 2 & 0 \\
    \end{pmatrix},
  \end{multline*}
  where $X = i(C_0,c_{i_1})$ and $Y = i(C_0,c_{i_2})$ and where the
  last relation is the equivalence under column and row operations.
  The upper left $3\times 3$ block is $i(C,C)$, and the lower right
  $3\times 3$ block is invertible. Hence $\rank(i(C''',C''')) =
  \rank(i(C,C)) + 3$.

  The calculation is analogous for $i(C'',C'')$ and $i(C',C')$. In
  the first case, we have the invertible matrix $
  \begin{pmatrix}
    0 & 2 \\
    2 & 0 \\
  \end{pmatrix}$
  in the lower right corner. In the second case, the lower right
  corner is a single zero entry, hence
  $\rank(i(C',C')) = \rank(i(C,C))$.
\end{proof}

\begin{lemma}\label{lemma:adding-a-puncture}
  Let $C$ be a filling collection of inconsistently marked
  simple closed curves on a surface $S$ with at least one puncture.
  Suppose that the curves of $C$ are in pairwise minimal position.

  Then there is a point $p \in S -C$ and marked simple closed curves
  $d$ and $e$ on $S - \{p\}$ such that $C' = C \cup \{d\}$ and
  $C'' = C \cup \{d,e\}$ are filling inconsistently marked
  collections on $S-\{p\}$, and the curves in each collection are in pairwise
  minimal position. Moreover, we have
  \begin{displaymath}
    \rank(i(C,C)) = \rank(i(C',C')) = \rank(i(C'',C'')) - 2.
  \end{displaymath}
\end{lemma}

\begin{proof}
  \begin{figure}[htb]
    \labellist
    \small\hair 2pt
    \pinlabel {$d$} [ ] at 40 16
    \pinlabel {$e$} [ ] at 63 6
    \pinlabel {$R$} [ ] at 23 6
    \endlabellist
    \centering
    \includegraphics[width=7cm]{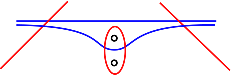}
    \caption{Creating two new curves by duplicating a puncture.}
    \label{fig:duplicating-puncture}
  \end{figure}
  Let $R$ be a component of $S-C$ which is a once-punctured disk. Let
  $p \in R$. Let $e$ be a curve surrounding the puncture and $p$ inside $R$.
  Let $d$ be a curve $c$ obtained from a curve on the boundary of $R$ by
  pulling it over $p$. (See \Cref{fig:duplicating-puncture}.) The properties of
  inconsistent marking, filling and minimal position are easy to verify. In
  addition, we have
  \begin{displaymath}
    i(C'',C'') =
    \begin{pmatrix}
      i(C_0,C_0) & \mathbf{x} & \mathbf{x} & 0 \\
      \mathbf{x}^T & 0 & 0 & 0 \\
      \mathbf{x}^T & 0 & 0 & 2 \\
      0 & 0 & 2 & 0 \\
    \end{pmatrix},
  \end{displaymath}
  where $C_0 = C - {c}$ and $\mathbf{x} = i(C_0,c)$. Note that the
  $i(C',C')$ is the submatrix obtained by deleting the last row and
  last column, and $i(C,C)$ is the submatrix obtained by deleting the
  last two rows and the last two columns. This proves the equation
  about the ranks.
\end{proof}

We are now ready to give the proof of \Cref{prop:curves_nonor}.

\begin{proof}[Proof of \Cref{prop:curves_nonor}]
  Since the constructions in the different cases have different flavors,
  we divide the proof into three parts. First we give examples for completely
  left-to-right collections of curves, then for unrestricted collections of
  curves when $g\ge 5$ and $g \le 4$, respectively.

  \begin{proofpart}[Completely left-to-right collections]
    \label{part:left-to-right}
    For any $g\ge 4$, $n\ge 0$ and $3\le r \le g-1$, we need to construct a
    filling and inconsistently marked completely left-to-right collection of
    curves on $N_{g,n}$ whose intersection matrix has rank $r$.

    \begin{figure}[ht]
      \centering
      \includegraphics[width=0.4\textwidth]{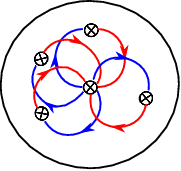}\quad
      \includegraphics[width=0.4\textwidth]{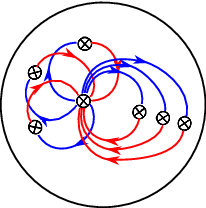}
      \caption{Rank 4 left-to-right collections on $N_5$ and $N_7$.
        Red arcs always cross blue arcs from the left to right.}
      \label{fig:nonor-left-to-right}
    \end{figure}

    When $r = g-1$ and $n = 0$, we think about $N_g$ and a sphere with $g$ crosscaps attached. We arrange $r$ curves around a central crosscap
    as shown on the left on \Cref{fig:nonor-left-to-right}. When the curves are
    marked as shown on \Cref{fig:nonor-left-to-right}, we obtain a completely
    left-to-right collection, because red arcs always cross the blue arcs from
    left to right (see the paragraphs before \Cref{prop:curves_nonor} for the
    explanation of the coloring convention). The intersection matrix is the
    $r\times r$ square matrix whose off-diagonal entries are 1 and whose
    diagonal entries are 0. Note that this matrix is invertible for all $r$:
    the inverse is the matrix whose off-diagonal entries are $1/(r-1)$ and
    whose diagonal entries are $-(r-2)/(r-1)$.

    For the cases $r < g-1$ and $n = 0$, we modify the previous arrangement by
    attaching crosscaps and replacing one of the curves with a set of new
    curves as on the right of \Cref{fig:nonor-left-to-right}. A bit of care
    should be taken here as the number of new curves is less than the number of
    attached crosscaps, whereas it may seem first to the eyes used to orientable
    surfaces that the two numbers are equal. This is because the curves that
    connect at the central crosscap are the antipodal ones. In the particular
    case shown by \Cref{fig:nonor-left-to-right}, the new set of curves
    consists of two curves, not three.

    We claim that this modification does not change the rank of the
    intersection matrix. For this, note that the intersection number of the new
    curves with the rest of the curves are proportional, so the corresponding
    columns and rows in the intersection matrix are scalar multiples of each
    other. Hence the rank of the intersection matrix is the same for the two
    examples on \Cref{fig:nonor-left-to-right}.

    When $n>0$, we consider the collection $C$ already constructed in the case
    $n = 0$, add $n$ disjoint isotopic copies of one of the curves in $C$, and
    arrange the $n$ punctures between the $n+1$ isotopic curves so that no two
    of them are isotopic in the punctured surface. These $n+1$ curves have the
    same intersection numbers with the other curves, hence the rank of the
    intersection matrix remains the same as in the case $n=0$. It is also clear
    that the new collection of curves remains completely left-to-right if the
    duplicated curves inherit their marking from the original curve.
  \end{proofpart}

  \begin{proofpart}[Unrestricted case, $g\ge 5$]
    Note that attaching a crosscap to the orientable surface $S_{g',n}$ yields the
    nonorientable surface $N_{2g'+1,n}$. Hence we can obtain nonorientable
    surfaces of odd (resp.~even) genus by attaching one (resp.~two) crosscap(s) to
    an orientable surface.

    During the proof of \Cref{prop:examples_of_filling_pairs}, we constructed
    for many $g',n,r'$ a collection of curves $C_{g',n,2r'}$ on $S_{g',n}$ such
    that $\rank(i(C_{g',n,2r'},C_{g',n,2r'})) = 2r'$. If $g' \ge 2$ and
    $r' \ge 2$, then $C_{g',n,2r'}$ has at least two complementary regions that
    allow applying \Cref{lemma:attaching_a_crosscap}: both of the two left-most
    regions work on \Cref{fig:maximal_filling,fig:figure_eight_filling}. In
    fact, \Cref{lemma:attaching_a_crosscap} can be applied subsequently for the
    two regions. That is, after applying it for the first region, the
    hypotheses of the lemma still hold for the other region. This gives
    examples for all triples $(g,n,r)$ where $g\ge 5$, $r \ne 3,5$ and $n$ is
    arbitrary.

    \begin{figure}[htb]
      \labellist
      \small\hair 2pt
      \pinlabel {$a_1$} [ ] at 15 43
      \pinlabel {$a_2$} [ ] at 88 43
      \pinlabel {$b_1$} [ ] at 48 51
      \pinlabel {$b_2$} [ ] at 130 49
      \pinlabel {$c$} [ ] at 101 63
      \endlabellist
      \centering
      \includegraphics[width=8cm]{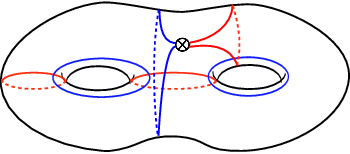}
      \caption{A rank 5 collection on $N_{5}$.}
      \label{fig:N_5_rank_5}
    \end{figure}
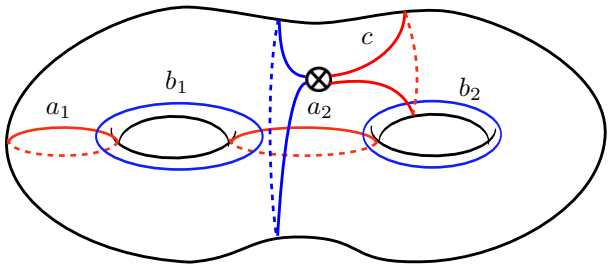

    Examples for the cases
    \begin{itemize}
    \item $g\ge 5$, $r=3$, $n$ is arbitrary,
    \item $g\ge 6$, $r=5$, $n$ is arbitrary,
    \end{itemize}
    have been given in \Cref{part:left-to-right}. The only case that
    remains is $g=5$, $r=5$. \Cref{fig:N_5_rank_5} gives an example
    when $n=0$. The intersection matrix is
    \begin{displaymath}
      i(C,C) =
      \begin{pmatrix}
        0 & 0 & 1 & 0 & 0 \\
        0 & 0 & 1 & 1 & 2 \\
        1 & 1 & 0 & 0 & 0 \\
        0 & 1 & 0 & 0 & 1 \\
        0 & 2 & 0 & 1 & 0 \\
      \end{pmatrix}
    \end{displaymath}
    and it is invertible. When $n>0$, we add parallel curves separating
    the punctures as in the end of \Cref{part:left-to-right}.
  \end{proofpart}

  \begin{proofpart}[Unrestricted case, $g\le 4$]
    \Cref{table:rank-ranges} summarizes the ranks we need to realize on the
    surfaces $N_{g,n}$ when $g \le 4$.

    \begin{table}[ht]
      \centering
      \begin{tabular}{c|cccc}
        n$\backslash$g & 1 & 2 & 3 & 4 \\
        \hline
        0   & $\emptyset$ & $\emptyset$ & $\emptyset$ & 3--5 (3,4,5)\\
        1   & $\emptyset$ & $\emptyset$ & 3--4 (3,4)& 3--8 (8)\\
        2   & $\emptyset$ & 3--4 (3,4)& 3--7 (7) & 3--10 \\
        3   & 3 (3) & 3--6 & 3--9 & 3--12 \\
        4   & 3--5 (4) & 3--8 & 3--11 & 3--14 \\
        5   & 3--7 & 3--10 & 3-13 & 3--16\\
      \end{tabular}
      \caption{Ranks to realize on nonorientable surfaces of genus at most 4.
        We give explicit examples of curve collections realizing the ranks
        shown in the parentheses and construct the other examples of using
        \Cref{lemma:attaching_a_crosscap,lemma:adding-a-puncture}.}
      \label{table:rank-ranges}
    \end{table}

    We construct only finitely many examples (shown in parentheses in
    \Cref{table:rank-ranges}). When $g \le 3$, these examples and
    \Cref{lemma:adding-a-puncture} take care of all cases. What is different in
    the case $g=4$ is that the surface with the smallest number of punctures
    ($N_4$) is closed, so we cannot apply \Cref{lemma:adding-a-puncture} to construct examples on $N_{4,1}$ from the examples on $N_4$. However, the rank
    3, 4 and 5 collections we will construct on $N_4$ still fill when a puncture is
    added, so the same collections realize the ranks 3, 4 and 5 on
    $N_{4,1}$. To realize 6 and 7, we apply
    \Cref{lemma:attaching_a_crosscap} for our rank 4 collections on
    $N_{3,1}$. Finally, we describe a rank 8 example explicitly.
    \Cref{lemma:adding-a-puncture} can then be used to complete the
    construction for all $N_{g,n}$ with $g=4$ and $n>1$.

    To complete the proof, we now give the examples for the cases listed in the parentheses in \Cref{table:rank-ranges}. We recall that other than the filling property, the collections of curves also need to be marked inconsistently. In other words, a blue and red arc should meet at every intersection.
    \begin{case}[$g=1$]
      \Cref{fig:genus-one} shows collections with
      \begin{displaymath}
        i(C,C) =
        \begin{pmatrix}
          0 & 2 & 2 \\
          2 & 0 & 2 \\
          2 & 2 & 0 \\
        \end{pmatrix}
        \quad \mbox{and} \quad
        i(C,C) =
        \begin{pmatrix}
          0 & 0 & 2 & 2 \\
          0 & 0 & 2 & 0 \\
          2 & 2 & 0 & 2 \\
          2 & 0 & 2 & 0 \\
        \end{pmatrix}.
      \end{displaymath}
      It is easy to see that both matrices are invertible.
      \begin{figure}[!htb]
        \centering
        \includegraphics[width=0.3\textwidth]{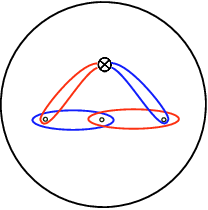}
        \quad
        \includegraphics[width=0.3\textwidth]{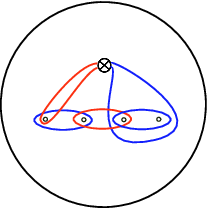}
        \caption{A rank 3 collection on $N_{1,3}$ and a rank 4
          collection on $N_{1,4}$. The surfaces are represented as spheres with added crosscaps and punctures.}
        \label{fig:genus-one}
      \end{figure}
    \end{case}

    \begin{case}[$g=2$]
      \Cref{fig:genus-two} shows collections with
      \begin{displaymath}
        i(C,C) =
        \begin{pmatrix}
          0 & 2 & 2 \\
          2 & 0 & 4 \\
          2 & 4 & 0 \\
        \end{pmatrix}
        \quad \mbox{and} \quad
        i(C,C) = \begin{pmatrix}
          0 & 2 & 2 & 2 \\
          2 & 0 & 2 & 2 \\
          2 & 2 & 0 & 4 \\
          2 & 2 & 4 & 0 \\
        \end{pmatrix}.
      \end{displaymath}
      Again, one can check that both matrices are invertible.

      \begin{figure}[ht]
        \centering
        \includegraphics[width=0.3\textwidth]{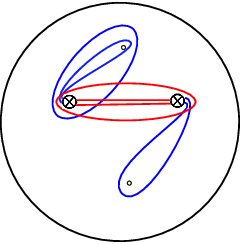}
        \quad
        \includegraphics[width=0.3\textwidth]{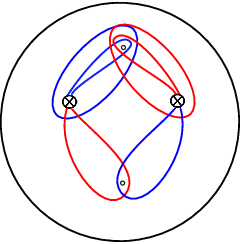}
        \caption{A rank 3 and rank 4 collection on $N_{2,2}$.}
        \label{fig:genus-two}
      \end{figure}
    \end{case}

    \begin{case}[$g=3$]
      \Cref{fig:genus-three-one-puncture} shows a filling collection
      with intersection matrix
      \begin{equation}\label{eq:genus-three-matrix}
        i(C,C) =
        \begin{pmatrix}
          0 & 0 & 1 & 0\\
          0 & 0 & 1 & 2\\
          1 & 1 & 0 & 1\\
          0 & 2 & 1 & 0\\
        \end{pmatrix},
      \end{equation}
      which has rank 4. Note that there is a complementary region with
      two disjoint curves on its boundary, hence
      \Cref{lemma:attaching_a_crosscap} indeed applies to yield rank 4, 6
      and 7 collections on $N_{4,1}$.

      By dropping the curve surrounding the hole which intersects only
      one other curve, the remaining three curves still fill and the
      intersection matrix is the lower right $3\times 3$ submatrix of
      \Cref{eq:genus-three-matrix}, which has rank 3.
      \begin{figure}[ht]
        \centering
        \includegraphics[width=0.5\textwidth]{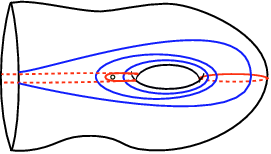}
        \caption{A rank 4 collection on $N_{3,1}$. Here the surface $N_{3,1}$ is represented by identifying antipodal points of the boundary component of the pictured orientable surface.}
        \label{fig:genus-three-one-puncture}
      \end{figure}

      On the left, \Cref{fig:genus-three-two-punctures-genus-four} shows a filling collection of curves on $N_{3,2}$ with intersection matrix
      \begin{displaymath}
        i(C,C) =
        \begin{pmatrix}
          0 & 2 & 2 & 2 & 2 & 2 & 4 \\
          2 & 0 & 0 & 2 & 4 & 4 & 4 \\
          2 & 0 & 0 & 2 & 4 & 2 & 2 \\
          2 & 2 & 2 & 0 & 2 & 2 & 4 \\
          2 & 4 & 4 & 2 & 0 & 0 & 4 \\
          2 & 4 & 2 & 2 & 0 & 0 & 2 \\
          4 & 4 & 2 & 4 & 4 & 2 & 0 \\
        \end{pmatrix},
      \end{displaymath}
      which has rank 7.
      \begin{figure}[ht]
        \centering
        \includegraphics[width=0.3\textwidth]{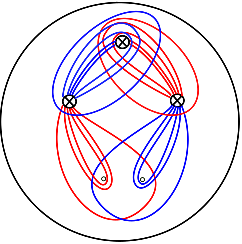}
        \quad
        \includegraphics[width=0.3\textwidth]{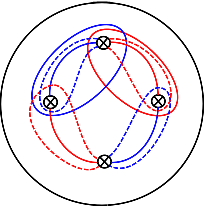}
        \caption{A rank 7 collection on $N_{3,2}$ and a rank 5 collection on $N_{4}$.}
        \label{fig:genus-three-two-punctures-genus-four}
      \end{figure}
    \end{case}

    \begin{case}[$g=4$]
      On the right, \Cref{fig:genus-three-two-punctures-genus-four} shows a filling collection of curves on $N_4$ with intersection matrix
      \begin{equation}\label{eq:genus-four-matrix}
        i(C,C) =
        \begin{pmatrix}
          0 & 2 & 2 & 2 & 2 \\
          2 & 0 & 2 & 2 & 2 \\
          2 & 2 & 0 & 2 & 2 \\
          2 & 2 & 2 & 0 & 4 \\
          2 & 2 & 2 & 4 & 0 \\
        \end{pmatrix},
      \end{equation}
      which has rank 5. By dropping both or one of the dashed curves,
      the remaining three or four curves still fill. The intersection
      matrices are the upper left $3\times 3$ and $4 \times 4$
      submatrices of \Cref{eq:genus-four-matrix}, which have rank 3 and
      rank 4, respectively.

      Finally, \Cref{fig:genus-four-one-punctures} shows a filling collection of curves on $N_{4.1}$ with intersection matrix
      \begin{displaymath}
        i(C,C) =
        \begin{pmatrix}
          0 & 2 & 2 & 2 & 2 & 2 & 4 & 0 \\
          2 & 0 & 2 & 2 & 2 & 2 & 4 & 0 \\
          2 & 2 & 0 & 4 & 4 & 4 & 8 & 0 \\
          2 & 2 & 4 & 0 & 0 & 0 & 0 & 0 \\
          2 & 2 & 4 & 0 & 0 & 2 & 2 & 2 \\
          2 & 2 & 4 & 0 & 2 & 0 & 2 & 2 \\
          4 & 4 & 8 & 0 & 2 & 2 & 0 & 4 \\
          0 & 0 & 0 & 0 & 2 & 2 & 4 & 0 \\
        \end{pmatrix}.
      \end{displaymath}
      It is easy to zero out the upper right and lower left $4\times 4$
      block using row and column operations. The remaining $4\times 4$
      matrices are invertible, hence $i(C,C)$ has rank 8.
      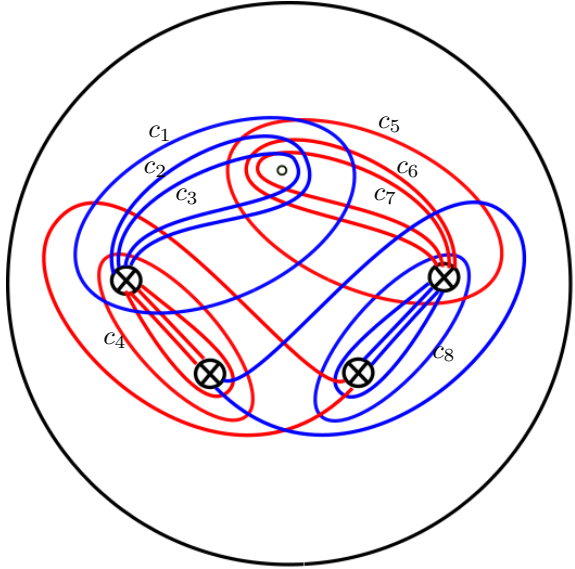
\begin{figure}[ht]
        \labellist
        \small\hair 2pt \pinlabel {$c_1$} [ ] at 34 96 \pinlabel {$c_2$}
        [ ] at 33 88 \pinlabel {$c_3$} [ ] at 40 82 \pinlabel {$c_4$} [
        ] at 24 50 \pinlabel {$c_5$} [ ] at 85 98 \pinlabel {$c_6$} [ ]
        at 89 88 \pinlabel {$c_7$} [ ] at 84 82 \pinlabel {$c_8$} [ ] at
        97 47
        \endlabellist
        \centering
        \includegraphics[width=0.5\textwidth]{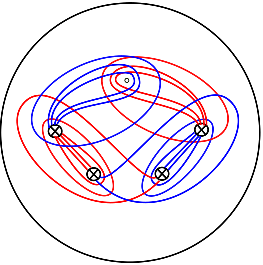}
        \caption{A rank 8 collection on $N_{4,1}$.}
        \label{fig:genus-four-one-punctures}
      \end{figure}
    \end{case}
  \end{proofpart}
  This completes the proof of \Cref{prop:curves_nonor}.
\end{proof}

\section{Two pseudo-Anosov examples for the sporadic cases}
\label{sec:two_examples}

We have tried but were not able to construct inconsistently marked collections
of curves with rank 6 intersection matrix on $N_4$ and rank 5 intersection
matrix on $N_{3,1}$. Since we have made many attempts using different perspectives and
were not successful, we conjecture that such collections of curves do not
exist. However, we will show that degree 6 and 5 stretch factors still exist on
the surfaces $N_4$ and $N_{3,1}$.

Due to the lack of suitable collections of curves, these constructions will not use Penner's construction. Instead, we will describe the mapping classes explicitly and compute the stretch factors using transition matrices of train track maps.

\begin{proposition}\label{prop:two-examples}
  We have $6 \in D(N_{4})$ and $5 \in D(N_{3,1})$.
\end{proposition}
\begin{proof}
  We will now represent the surface $N_{3,1}$ as a polygon with its sides identified. On \Cref{fig:tt_five}, the open circles on the boundary of each rectangle divide the boundary to six segments. The two parallel pairs of side which are not marked by an arrow are identified by a translation. The pair marked by an arrow is identified by a flip. The six open circles on the boundary identify to a single point, representing the puncture.

  The first picture of \Cref{fig:tt_five} shows a train track $\tau$
  on this surface. The train track has two switches, seven branches and
  four complementary regions. Three of these regions is a trigon, one
  is a monogon.

\begin{figure}[htb]
    \centering
    \includegraphics[scale=0.7]{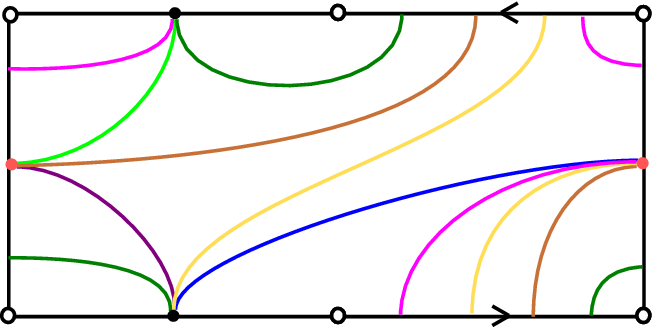}
    \includegraphics[scale=0.7]{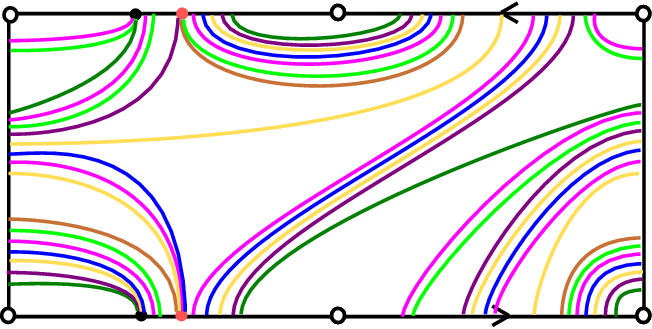}
    \caption{A train track embedded into $N_{3,1}$ in two different
      ways. }
    \label{fig:tt_five}
  \end{figure}

  The second picture shows the same train track, embedded in $N_{3,1}$
  in a different way. The map from $\tau$ to this second train track
  extends to a homeomorphism of the surface which well-defined up to
  homotopy. In other words, \Cref{fig:tt_five} describes some
  $f \in \Mod(N_{3,1})$ and the second train track is $f(\tau)$.

  Note that $f(\tau)$ is carried on $\tau$ and the
  matrix describing the $f$-action on the branches of $\tau$ is
  \begin{displaymath}
    \begin{pmatrix}
      1 & 1 & 1 & 1 & 0 & 0 & 1 \\
      1 & 1 & 0 & 1 & 0 & 0 & 0 \\
      0 & 1 & 0 & 1 & 0 & 0 & 1 \\
      0 & 1 & 0 & 1 & 1 & 0 & 0 \\
      0 & 0 & 1 & 1 & 0 & 0 & 1 \\
      0 & 0 & 0 & 1 & 0 & 0 & 0 \\
      0 & 0 & 1 & 1 & 0 & 1 & 1 \\
    \end{pmatrix}.
  \end{displaymath}

  This matrix is Perron--Frobenius and the minimal polynomial of the
  Perron--Frobenius eigenvalue is $x^5-3x^4-x^3+x^2-x-1$. Hence $f$ is
  a pseudo-Anosov mapping class with degree 5 stretch factor.

\begin{figure}[htb]
    \centering
    \includegraphics[scale=0.6]{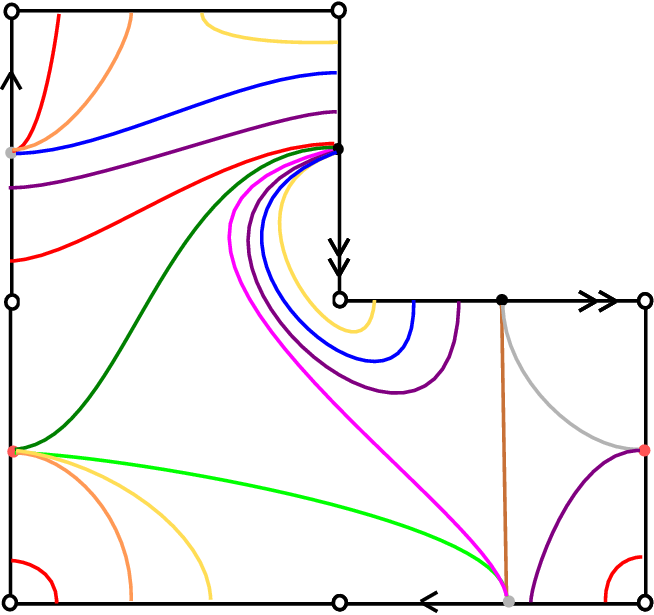}
    \includegraphics[scale=0.6]{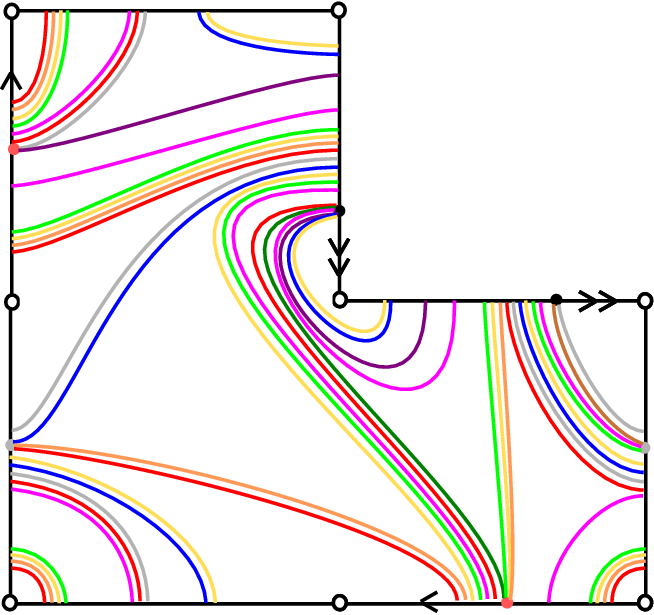}
    \caption{A train track embedded into $N_{4,1}$ in two different
      ways. Segments of the boundaries of the polygons are identified
      by translations except for the sides marked by arrows. The open
      circles on the boundary identify to a single point, representing
      the puncture. }
    \label{fig:tt_six}
  \end{figure}

  For our second example, consider the pictures on
  \Cref{fig:tt_six} which describe an element of
  $\Mod(N_{4,1})$ in the same way as in the previous example. The matrix describing the action on the train track is
  \begin{displaymath}
    \begin{pmatrix}
      1 & 0 & 1 & 1 & 0 & 0 & 0 & 0 & 1 & 1 \\
      0 & 0 & 1 & 0 & 1 & 0 & 0 & 0 & 0 & 1 \\
      0 & 1 & 1 & 1 & 1 & 0 & 0 & 0 & 1 & 0 \\
      0 & 1 & 0 & 1 & 0 & 1 & 0 & 0 & 1 & 0 \\
      0 & 0 & 1 & 1 & 0 & 0 & 0 & 0 & 1 & 1 \\
      0 & 0 & 0 & 0 & 0 & 0 & 0 & 1 & 0 & 0 \\
      0 & 0 & 0 & 1 & 0 & 0 & 0 & 0 & 0 & 0 \\
      1 & 0 & 0 & 0 & 0 & 0 & 1 & 0 & 1 & 0 \\
      0 & 1 & 0 & 0 & 0 & 0 & 1 & 0 & 2 & 0 \\
      0 & 0 & 0 & 0 & 0 & 0 & 0 & 0 & 1 & 0 \\
    \end{pmatrix}
  \end{displaymath}
  This is again Perron--Frobenius and the minimal polynomial of the
  largest eigenvalue is $x^6 - 3x^5 - x^4 - x^2 + x + 1$. Hence the
  mapping class is pseudo-Anosov with degree 6 stretch factor.

  We are not done, because we need a degree 6 example on $N_4$, not on
  $N_{4,1}$. But note that the complementary region of the train track
  containing the puncture is a bigon, hence our pseudo-Anosov map has
  a 2-pronged singularity at the puncture. So we can fill in the
  puncture to obtain a pseudo-Anosov map on $N_4$ with the same
  degree 6 stretch factor.
\end{proof}

  We remark that the two candidate mapping classes, each as a product of three
  Dehn twists, were found with the help of the computer software \texttt{flipper}
  by Mark Bell \cite{flipper}. The train track maps and the transition matrices
  were then computed entirely by hand. We have mentioned this in the proof already but we would like to reiterate that verifying the correctness of the examples is a much easier process than finding them, and this process can also be done---and have been done---without using computers: one only needs to check that in each case the two train tracks that we claim to be isomorphic are indeed isomorphic and then use the pictures to verify that the transition matrices are also as shown above.

\section{Proofs of the main theorems}
\label{sec:proofs}

In this section, we give the proofs of the main theorems.

\subsection{Bounds on the algebraic degree}
\label{sec:bounds}

Determining the sets $D(S)$ and $D^+(S)$ has two parts: ruling out the integers that do not arise as degrees and constructing examples realizing integers that have not been ruled out. The first part has almost entirely been done before, this section is for summarizing the relevant results.

\paragraph{General bounds.}
The first upper bound on the degree is due to Thurston \cite{Thurston88}. Recall that $\Teich(S)$ stands for Teichm\"uller space of the surface $S$.

\begin{proposition}\label{prop:bound}
  For any finite type surface $S$, we have
  \begin{displaymath}
    \max D(S) \le \dim(\Teich(S)).
  \end{displaymath}
\end{proposition}

Next, we give a bound on the largest degree in $D^+(S)$. This was certainly also known before, but since we have not found a proof a literature, we include one here.

First we need a fact about degrees of powers of algebraic numbers.
The algebraic degree of some algebraic numbers---for example,
$\sqrt{2}$---decreases under taking powers. We will use the fact that
this does not happen for pseudo-Anosov stretch factors.

\begin{lemma}\label{lemma:degree-under-powers}
  If $\lambda$ is a pseudo-Anosov stretch factor, then $\deg(\lambda) =
  \deg(\lambda^k)$ for every positive integer $k$.
\end{lemma}
\begin{proof}
  Let
  $p(x) = (x-\lambda)(x-\lambda_2) \cdots (x-\lambda_n) \in \QQ[x]$ be
  a minimal polynomial of $\lambda$, factored over the complex numbers.
  Note that $n = \deg(\lambda)$.

  Consider the polynomial
  $p_k(x) = (x-\lambda^k)(x-\lambda_2^k) \cdots (c-\lambda_n^k)$. The
  group $G = Aut(\QQ(\lambda)/\QQ)$ of automorphisms of the number field
  $\QQ(\lambda)$ permutes the roots of $p_k(x)$, hence acts trivially on the coefficients. It
  follows that $p_k(x) \in \QQ[x]$. We obtain that
  $\deg(\lambda^k) \le \deg(\lambda)$ for every positive integer $k$. This holds in general, we have not yet used the assumption that
  $\lambda$ is a pseudo-Anosov stretch factor.

  For the reverse inequality, it suffices to show that $p_k(x)$ is
  irreducible over $\QQ$ for every positive integer $k$ if $\lambda$
  is a pseudo-Anosov stretch factor. If this was not the case, the length
  of the $G$-orbit of $\lambda^k$ would be less than $n$, so there
  would be a nontrivial element of $G$ fixing $\lambda^k$. As a
  consequence, $|\lambda_i| = |\lambda|$ would hold for some
  $2 \le i \le n$. This is impossible, since $\lambda$ is the unique
  largest root of $p(x)$ in absolute value.
 \end{proof}

 The second lemma we will use relates degrees on punctured surfaces to degrees on closed surfaces. Recall that $\overline{S}$ denotes the closed surface obtained from $S$ by filling in the punctures.

\begin{lemma}\label{lemma:punctures-dont-matter}
  For any finite type surface $S$, we have $D^+(S) = D^+(\overline{S})$.
\end{lemma}
\begin{proof}
  If $\psi$ is a pseudo-Anosov map on $S$ with an orientable invariant
  foliation, then it has no 1-pronged singularities. So the
  pseudo-Anosov map on $\overline{S}$ obtained by extending $\psi$ to
  the punctures is still pseudo-Anosov and has the same stretch factor
  as $\psi$. This shows that $D^+(S) \subset D^+(\overline{S})$.

  If $\psi$ is a pseudo-Anosov map on $\overline{S}$ with stretch
  factor $\lambda$, then some power of it has at least as many fixed points as
  the number of punctures of $S$. Hence there is a pseudo-Anosov map
  on $S$ whose stretch factor is some power of $\lambda$. By
  \Cref{lemma:degree-under-powers}, the algebraic degree of $\lambda$
  is preserved under powers, hence
  $D^+(\overline{S}) \subset D^+(S)$.
\end{proof}

\begin{proposition}\label{prop:plus-bound}
  For any finite type surface $S$, we have
  \begin{displaymath}
    \max D^+(S) \le \dim(H^1(\overline{S},\RR)).
  \end{displaymath}
\end{proposition}
\begin{proof}
  By \Cref{lemma:punctures-dont-matter} it suffices to show that
  $\max D^+(S) \le \dim(H^1(S,\RR))$ if $S$ is closed. This follows
  from the fact that the orientable invariant foliation corresponds to
  a 1-form on $S$, so it is an eigenvector of the action of the
  $H^1(S,\RR)$.
\end{proof}

\paragraph{Bounds on odd degrees.} The next group of results is on restrictions on odd degrees.

\begin{proposition}\label{prop:odd-implies-less-than-half}
  If $S$ is orientable, $d \in D(S)$ and $d$ is odd, then
  $d \le \frac12 \dim(\Teich(S))$.
\end{proposition}

For closed surfaces, this result is due to Long \cite[Theorem
3.3]{Long85}. McMullen later gave an different proof \cite[Theorem
10]{Shin16} which works for punctured surfaces as well. An
analogous argument (where the action
on the space of projective measured foliations is replaced by
the action on homology) also yields the following.

\begin{proposition}\label{prop:odd-implies-less-than-half-or}
  If $S$ is orientable, $d \in D^+(S)$ and $d$ is odd, then $d \le
  \frac12 \dim(H^1(\overline{S},\RR))$.
\end{proposition}

\Cref{prop:odd-implies-less-than-half,prop:odd-implies-less-than-half-or}
have no analogs for nonorientable surfaces as demonstrated by
\Cref{prop:curves_nonor}.

\paragraph{Two facts about nonorientable surfaces.}

Degree two stretch factors occur on the torus and the four times punctured sphere, so they arise on all higher complexity orientable surfaces as well via puncturing and branched coverings. In sharp contrast, degree two stretch factors do not occur on nonorientable surfaces at all. This requires only a small observation, but as far as we know, this fact has not been noticed before.
\begin{proposition}\label{two-not-degree}
  If $S$ is nonorientable, then $2\notin D(S)$.
\end{proposition}
\begin{proof}
  If $\deg(\lambda) = 2$, then its minimal polynomial has the form
  $x^2 \pm kx \pm 1$ for some $k \in \ZZ$, since $\lambda$ in
  algebraic unit. But this is impossible, since $\pm 1/\lambda$ is
  never a Galois conjugate of $\lambda$ when the pseudo-Anosov map is
  supported on a nonorientable surface. (This is stated for
  $1/\lambda$ in \cite[Proposition 2.3]{StrennerSAF}, but the
  same proof works for $-1/\lambda$ as well.)
\end{proof}

Finally, we include the following well-known fact for completeness.

\begin{proposition}\label{prop:N3-no-pa}
  The closed nonorientable surface of genus 3 does not admit
  pseudo-Anosov maps.
\end{proposition}
\begin{proof}
  There is a unique one-sided curve on the surface whose complement is
  a one-holed torus \cite[Lemma 2.1]{Scharlemann82}. This curve is
  fixed by every mapping class.
\end{proof}

\subsection{Proof of the degree theorem}
\label{sec:degree_proof}

We are now ready to state and prove a more general version of
\Cref{theorem:degrees_simple}. We use the notations $[a,b]_\even$ and
$[a,b]_\odd$ from \Cref{theorem:degrees_simple}, and abuse the interval
notation $[a,b]$ to mean the set of integers in the corresponding real
interval.

\begin{theorem}\label{theorem:degrees-general}
  Let $S$ be a finite type orientable surface.
  \begin{enumerate}[(i)]
  \item\label{item:Dp_or} We have
    \begin{equation}\label{eq:degrees-plus}
      D^+(S) = \big[2,\dim(H_1(\overline{S},\RR))\big]_\even \cup
      \left[3,\frac12 \dim(H_1(\overline{S},\RR))\right]_\odd.
    \end{equation}
  \item\label{item:D_or} If $S$ has an even number of punctures or
    $\frac12 \dim(\Teich(S))$ is even, then
    \begin{equation}\label{eq:degrees-good-case}
      D(S) = \big[2,\dim(\Teich(S))\big]_\even \cup \left[3,\frac12 \dim(\Teich(S))\right]_\odd.
    \end{equation}
  \item\label{item:D_or_odd} If $S$ has an odd number of punctures and
    $\frac12 \dim(\Teich(S))$ is odd, then either
    \Cref{eq:degrees-good-case} holds or
    \begin{equation}\label{eq:degrees-bad-case}
      D(S) = \big[2,\dim(\Teich(S))\big]_\even \cup \left[3,\frac12 \dim(\Teich(S))-1\right]_\odd.
    \end{equation}
  \end{enumerate}
  Let $N$ be a finite type nonorientable surface.
  \begin{enumerate}[(i)]
    \setcounter{enumi}{3}
  \item\label{item:Dp_nonor} We have
    \begin{displaymath}
      D^+(N) = \big[3,\dim(H_1(\overline{N},\RR))\big].
    \end{displaymath}
  \item\label{item:D_nonor} If $N$ is not the closed
    nonorientable surface $N_3$ of genus 3, then
    \begin{displaymath}
      D(N) = \big[3,\dim(\Teich(N))\big].
    \end{displaymath}
  \end{enumerate}
\end{theorem}
The reason why the cases (\ref{item:D_or}) and (\ref{item:D_or_odd})
are separated is that we cannot realize the odd degree
$\frac12 \dim(\Teich(S))$ when the surface is not a double cover of a
nonorientable surface. It seems possible that realizing this degree is
not possible using Penner's construction. We do not know if it is
possible using other constructions. Randomized computer experiments
yield almost exclusively even degree stretch factors, so searching for
these degrees with computers also seems a nontrivial task.

The exclusion of the surface $N_3$ is necessary, since
$\dim(\Teich(N_3)) = 3$, so the interval
$\left[3,\dim(\Teich(N_3))\right]$ is nonempty. However, this surface
does not admit pseudo-Anosov maps (\Cref{prop:N3-no-pa}).
\begin{proof}
  The fact that the sets cannot be larger than stated follows from the
  results in \Cref{sec:bounds}. It remains to prove that all the
  degrees claimed in the theorem can be realized.

  The statements (\ref{item:Dp_nonor}) and (\ref{item:D_nonor}) follow
  from \Cref{theorem:curves-and-degree} and \Cref{prop:curves_nonor} with
  the exception of the two sporadic cases: $6 \in D(N_4)$ and
  $5 \in D(N_{3,1})$. (Compare \Cref{table:nonor-pAs} and
  \Cref{table:nonor-curves}.) Examples for these two cases were given
  in \Cref{sec:two_examples}.

  \begin{table}[ht]
    \centering
    \begin{tabular}{c|ccccc}
      n$\backslash$g & 1 & 2 & 3 & 4 & 5 \\
      \hline
      0   & $\emptyset$ & $\emptyset$ & $\emptyset$ & A & A \\
      1   & $\emptyset$ & $\emptyset$ & A & A & A \\
      2   & $\emptyset$ & A & A & A & A \\
      3   & A & A & A & A & A \\
    \end{tabular}
    \caption{Nonorientable surfaces admitting (A) and not
      admitting ($\emptyset$) pseudo-Anosov maps.}
    \label{table:nonor-pAs}
  \end{table}

  The construction of even degrees for (\ref{item:Dp_or}),
  (\ref{item:D_or}) and (\ref{item:D_or_odd}) follow from
  \Cref{theorem:curves-and-degree} and
  \Cref{prop:examples_of_filling_pairs}.

  It remains to construct odd degrees on orientable surfaces. We obtain these by lifting pseudo-Anosov maps from nonorientable
  surfaces. Lifting preserves the stretch factor hence also the
  degree.

  Suppose $S$ is an orientable surface with an even number of
  punctures, and let $S \to N$ be a covering where $N$ is a
  nonorientable surface. We have $\dim(\Teich(S)) = 2 \dim(\Teich(N))$
  and $\dim(H_1(\overline{S})) = 2 \dim(H_1(\overline{N}))$. If
  $S \ne S_2$, then $N \ne N_3$, and it follows that
  $[3,\frac 12 \dim(\Teich(S))] \subset D(S)$ and
  $[3,\frac 12 \dim(H_1(\overline{S}))] \subset D^+(S)$. This proves
  \Cref{eq:degrees-plus} and \Cref{eq:degrees-good-case} when $S$ has
  an even number of punctures and $S \ne S_2$. In the case $S = S_2$,
  we only need to show that $3 \in D(S)$. This is shown by an example
  in Section 7 of \cite{Shin16}. Note that \Cref{lemma:punctures-dont-matter}
  implies that \Cref{eq:degrees-plus} holds also when $S$ has an odd
  number of punctures. This completes the proof of (\ref{item:Dp_or}).

  Finally, suppose $S$ is an orientable surface that has an odd
  number of punctures. Let $S'$ be the surface obtained from $S$ by filling in one puncture. We have already shown that \Cref{eq:degrees-good-case} holds for $S'$, and we will use this to show that it also holds (or almost holds) for $S$.

  To prove that \Cref{eq:degrees-good-case} holds for $S$ when $\frac12 \dim(\Teich(S))$ is even, it suffices to show that the left hand side contains the right hand side. Note that since $\frac12 \dim(\Teich(S))$ is even, there is no difference between the right hand sides of \Cref{eq:degrees-good-case} for $S$ and $S'$. On the other hand, the argument in the proof of \Cref{lemma:punctures-dont-matter} shows that $D(S') \subset D(S)$. So since $D(S')$ is already proven to contain the right hand side, the same is true for $D(S)$. This completes the proof of statement (\ref{item:D_or}).

  Now suppose $\frac12 \dim(\Teich(S))$ is odd. By an analogous argument, the fact that $D(S)$ contains the right hand side of \Cref{eq:degrees-bad-case} follows from the fact that $D(S') \subset D(S)$ and the fact that $D(S')$ contains the right hand side. This completes the proof of (\ref{item:D_or_odd}) and hence the proof of the theorem.
\end{proof}

\subsection{Proof of the trace field theorem}
\label{sec:trace_field_proof}

In this section, we give the proof of \Cref{theorem:trace_fields}. First we need the following lemma which will be used to show that the stretch factor $\lambda$ and its reciprocal $\lambda^{-1}$ tend to be Galois conjugates when they arise from Penner's construction on an orientable surface.

\begin{lemma}\label{lemma:reciprocal}
  Suppose the surface $S$ is orientable and let $\Omega$ be the
  intersection matrix of a collection of curves satisfying the
  hypotheses of Penner's construction. If $M$ is a product of the
  matrices $Q_i$ with Perron--Frobenius eigenvalue $\lambda$, then
  $\lambda^{-1}$ is also an eigenvalue of $M$.
\end{lemma}
\begin{proof}
  Recall from the proof of \Cref{prop:multiplicity-of-one} the linear map $\ell_M: \hV \to \hV$ induced by the left action of $M$ on $\hV = \RR^n/\Nul(\Omega)$. We will show that $\ell_M$ is a symplectic transformation. Since $\lambda$ is an eigenvalue of $\ell_M$ and eigenvalues of symplectic transformations come in reciprocal pairs \cite[Chapter 2]{MeyerHallOffin09}, this will complete the proof.

  Since $S$ is orientable, our collection of curves is a union of two
  multicurves $A$ and $B$. Therefore $\Omega$ has the block form
  \begin{displaymath}
    \Omega = \begin{pmatrix}
      0 & X \\
      X^T & 0 \\
    \end{pmatrix}
  \end{displaymath}
  where $X$ is an $a\times b$ matrix where $a$ and $b$ are the number
  of curves in $A$ and $B$, respectively.

  Define the alternating bilinear form
  $\langle \,\cdot\,,\,\cdot\, \rangle_\Delta$ on $\RR^n$ by the formula
  \begin{displaymath}
    \langle \vv_1,\vv_2 \rangle_\Delta = \vv_1^T\Delta \vv_2,
  \end{displaymath}
  where
  \begin{displaymath}
    \Delta =
    \begin{pmatrix}
      0 & X \\
      -X^T & 0 \\
    \end{pmatrix}.
  \end{displaymath}
  The matrices $\Omega$ and $\Delta$ are related by the equations
  $\Delta = U\Omega = -\Omega U$, where
  \begin{equation}\label{eq:U}
    U =
    \begin{pmatrix}
      I_a & 0 \\
      0 & -I_b \\
    \end{pmatrix}
  \end{equation}
  and $I_a$ and $I_b$ are the $a\times a$ and $b\times b$ identity
  matrices.

  Next we show that $\langle \,\cdot\,,\,\cdot\, \rangle_\Delta$
  descends to a symplectic form on $\hV$. If
  $\vv_0,\vv'_0 \in \Nul(\Omega)$, then
  \begin{displaymath}
    (\vv' + \vv'_0)^T \Delta (\vv + \vv_0) = \vv'^T\Delta \vv,
  \end{displaymath}
  because $\Delta \vv_0 = U\Omega \vv_0 = 0$ and
  $\vv'^T_0\Delta = -\vv'^T_0\Omega U = 0$. As a consequence, $\Delta$
  gives rise to a well-defined alternating form on $\hV$ that we also
  denote by $\langle \,\cdot\,,\, \cdot\, \rangle_\Delta$. This form
  is nondegenerate, because for any
  $\vv \in \RR^n - \Nul(\Omega)$, we have $\Delta\vv\ne 0$,
  therefore there exists $\vv' \in \RR^n$ with
  $\vv'^T\Delta \vv \ne 0$. Hence
  $\langle \,\cdot\,,\, \cdot\, \rangle_\Delta$ is a symplectic form
  on $\hV$.

  Finally, we show that the left actions of the $Q_i$ preserve this
  symplectic form on $\hV$. Using the fact that the diagonal matrices
  $U$ and $D_i$ commute, we have
  \begin{align*}
    Q_i^T \Delta Q_i &= (I+\Omega D_i)U \Omega (I + D_i \Omega) = \\
    &= U\Omega
    + \Omega D_i U\Omega + U\Omega D_i \Omega + \Omega D_i U \Omega D_i
    \Omega = \\ &= \Delta -U \Omega D_i \Omega + U \Omega D_i \Omega +
    \Omega U D_i \Omega D_i \Omega = \Delta,
  \end{align*}
  where the last term vanishes because the diagonal entries of
  $\Omega$ are zero, so $D_i \Omega D_i = 0$.
\end{proof}

\begin{proof}[Proof of \Cref{theorem:trace_fields}]
  All even degrees between 2 and $6g-6$ can be realized on $S_g$ as
  the algebraic degree of stretch factors by
  \Cref{theorem:degrees-general}. The construction of these examples
  occurs on the surfaces $S_g$ themselves (as opposed to odd degrees
  that are constructed by lifting from nonorientable surfaces). By
  \Cref{lemma:reciprocal}, both $\lambda_k$ and $\lambda_k^{-1}$ are
  eigenvalues of the matrices $M_{\gamma,k}$ in
  \Cref{theorem:construction-of-degree-r}, so $\lambda_k$ and
  $\lambda_k^{-1}$ are Galois conjugates if $k$ is large enough. It
  follows that $\QQ(\lambda) : \QQ(\lambda + \lambda^{-1}) = 2$, and
  we obtain that the trace field $\QQ(\lambda + \lambda^{-1})$ can
  have any degree from 1 to $3g-3$.

  On the other hand, the degree of the trace field cannot be larger
  than $3g-3$. If it were, then we would have $\QQ(\lambda) : \QQ(\lambda + \lambda^{-1}) = 1$, otherwise the degree of $\lambda$ would be bigger than $6g-6$, contradicting \Cref{prop:bound}. In other words, $\lambda$ and $\lambda^{-1}$ would not be Galois conjugates. But if they are not Galois conjugates, then
  McMullen's argument that proves
  \Cref{prop:odd-implies-less-than-half} would prove that
  $\deg(\lambda) \le 3g-3$.
\end{proof}

\bibliographystyle{alpha}
\bibliography{../mybibfile}

\end{document}